\documentclass[11pt,a4paper]{amsart}
\usepackage{amsmath,amssymb,amsthm,enumerate}
\usepackage{graphicx}
\usepackage{enumerate}
\title[Noncommutative LU-decomposition and Jacobi polynomials ]
{LU-decomposition of a noncommutative linear system
and Jacobi polynomials }
\author{Oscar Brega}
\address{CIEM-FaMAF, Universidad Nacional de C\'ordoba}
\email{brega@famaf.unc.edu.ar}

\author{Leandro Cagliero}
\address{CIEM-FaMAF, Universidad Nacional de C\'ordoba}
\email{cagliero@famaf.unc.edu.ar}

\date{}

\numberwithin{equation}{section}

\theoremstyle{plain}
\newtheorem{theorem} {Theorem} [section]
\newtheorem{lemma} [theorem] {Lemma}
\newtheorem{corollary} [theorem] {Corollary}
\newtheorem{proposition} [theorem] {Proposition}

\theoremstyle{definition}
\newtheorem{definition} [theorem] {Definition}
\newtheorem{remark} [theorem] {Remark}

\renewcommand \parallel {/\kern-3pt/}

\newcommand \N {\mathbb N}

\renewcommand \k { \textrm{k}}

\newcommand \lieg {\mathfrak{g}}
\newcommand \liek {\mathfrak{k}}

\newcommand \liea {\mathfrak{a}}

\newcommand \End {\operatorname{End}}

\newcommand \ad {\operatorname{ad}}

\newcommand \eq {\operatorname{eq}}
\newcommand \Eq {\operatorname{Eq}}

\begin{document}

\begin{abstract}
In this paper we obtain the LU-decomposition of a
noncommutative linear system of equations
that, in the rank one case, characterizes the image of the Lepowsky homomorphism
$U(\lieg)^{K}\to U(\liek)^{M}\otimes U(\liea)$.
This LU-decomposition can be transformed into very simple matrix identities,
where the entries of the matrices involved belong to a special class of Jacobi polynomials.
In particular, each entry of the L part of the original system is
expressed in terms of a single ultraspherical Jacobi polynomial.
In turns, these matrix identities yield a biorthogonality relation between
the ultraspherical Jacobi polynomials.
\end{abstract}

\maketitle


\section{Introduction}
\label{sec:intro}

\subsection{The noncommutative linear system}
Let $\k$ be a field of characteristic zero,
let $\mathcal{A}$ be an associative, not necessarily commutative, $\k$-algebra with unit
and let $\mathcal{M}$ be a unital $\mathcal{A}$-bimodule.

This paper is devoted to perform the gaussian elimination process on the
following non commutative homogeneous system of
infinitely many linear equations and infinite unknowns
\begin{equation}\label{eq:intro, first}
\begin{array}{lllllllllllll}
  E  p_0 & \!\!\!\!+ &\!\!\!\! E  p_1  & \!\!\!\!+ &\!\!\!\! E  p_2 &  \dots& =
                    & p_{0}E   & \!\!\!\!- &\!\!\!\! p_{1} E  & \!\!\!\!+ &\!\!\!\! p_{2}E &  \dots  \\[2.4mm]
  E^2p_0 & \!\!\!\!+ &\!\!\!\! 2E^2p_1 & \!\!\!\!+ &\!\!\!\! 2^2E^2p_2 &  \dots& =
                    & p_{0}E^2 & \!\!\!\!- &\!\!\!\! 2p_{1} E^2& \!\!\!\!+ &\!\!\!\! 2^2p_{2}E^2&  \dots \\[2.4mm]
  E^3p_0 & \!\!\!\!+ &\!\!\!\! 3E^3p_1 & \!\!\!\!+ &\!\!\!\! 3^2E^3p_2 &  \dots& =
                    & p_{0}E^3 & \!\!\!\!- &\!\!\!\! 3p_{1} E^3& \!\!\!\!+ &\!\!\!\! 3^2p_{2}E^3&  \dots \\[2mm]
\vdots && \vdots && \vdots &&& \vdots&& \vdots&& \vdots
\end{array},
\end{equation}
where $E$ is a given element of $\mathcal{A}$ and
$p_0,p_1,p_2,\dots$ belong to $\mathcal{M}$ and are the unknowns.

It is clear that $(p_0,p_1,\dots,p_d,0,\dots)$ is a solution of the system \eqref{eq:intro, first}
if and only if the polynomial $p=p_0+p_1t+p_2t^2+\dots+p_dt^d$, belonging to $\mathcal{M}[t]:=\mathcal{M}\otimes\k[t]$,
satisfies the equations
\begin{equation}\label{eq:intro B with H=0}
E^np(n)=p(-n)E^n  \quad  \text{for all} \,\, n\in\N.
\end{equation}

This system
can not be expressed as a single matrix
equation $AX=0$ with $A$ a matrix with coefficients in $\mathcal{A}$.
In fact, non commutative systems of (homogeneous) linear equations
can seldom be expressed as a single matrix equation with coefficients in $\mathcal{A}$.
Even when it is possible, only in
exceptional cases a gaussian elimination process can be performed
in a satisfactory way. For instance, if
$\dim_{\k}\mathcal{A}<\infty$ and $\dim_{\k}\mathcal{M}<\infty$,
by using bases of $\mathcal{A}$ and $\mathcal{M}$,
one can express every finite system of linear equations as a
single matrix equation $AX=B$ with coefficients in $\k$.
However the size of $A$ might be so big that this approach does
not help very much.
Some papers dealing with this subject are \cite{C}, \cite{CS}, \cite{Or}
or \cite{GGRW}.

On the other hand, by using the left and right
regular actions $L,R:\mathcal{A}\to\End_\k(\mathcal{M})$ of
$\mathcal{A}$ on $\mathcal{M}$,
it is indeed possible to express every  non commutative system of linear
equations as a single matrix equation but with coefficients in
$\End_\k(\mathcal{M})$.
In our case, the system \eqref{eq:intro, first}
can be expressed as single matrix equation $M_0X=0$, where
$X=\left(
\begin{smallmatrix}
p_0       \\[.7mm]
p_1       \\[.7mm]
p_2       \\[-.5mm]
\vdots
\end{smallmatrix}
\right)$
and
\begin{equation*}
M_0= { \left(
\begin{smallmatrix}
L_E   - R_E   \;&\; L_E    + R_E    \;&\; L_E    -  R_E   \;&\;  L_E   +   R_E   & .\;.\;.  \\[3mm]
L_E^2 - R_E^2 \;&\; 2L_E^2 + 2R_E^2 \;&\; 4L_E^2 - 4R_E^2 \;&\; 8L_E^2 +  8R_E^2 & .\;.\;.  \\[3mm]
L_E^3 - R_E^3 \;&\; 3L_E^3 + 3R_E^2 \;&\; 9L_E^3 - 9R_E^2 \;&\;27L_E^3 + 27R_E^2 & .\;.\;.  \\[0mm]
\vdots    & \vdots   & \vdots   & \vdots
\end{smallmatrix}
\right). }
\end{equation*}
In order to have an insight into the LU-decomposition of this
system assume for a moment that the unknown polynomial
has degree 2.
In this case it is easy to see that
\begin{multline*}
\left(
\begin{smallmatrix}
L_E   - R_E   \;&\; L_E    + R_E    \;&\; L_E    -  R_E    \\[3mm]
L_E^2 - R_E^2 \;&\; 2L_E^2 + 2R_E^2 \;&\; 4L_E^2 - 4R_E^2  \\[3mm]
L_E^3 - R_E^3 \;&\; 3L_E^3 + 3R_E^2 \;&\; 9L_E^3 - 9R_E^2
\end{smallmatrix}
\right) \\
=
\left(
\begin{smallmatrix}
1&0&0             \\[2mm]
L_{{E}}+R_{{E}}&1&0\\[2mm]
L_{{E}}^{2}+R_{{E}}L_{{E}}+R_{{E}}^{2}&2\,L_{{E}}+2\,R_{{E}}&1
\end{smallmatrix}
\right)
\left(
\begin{smallmatrix}
L_{{E}}-R_{{E}}&L_{{E}}+R_{{E}}&L_{{E}}-R_{{E}}\\[2mm]
0& \left( L_{{E}}-R_{{E}} \right) ^{2}&3\,L_{{E}}^{2}-3\,R_{{E}}^{2}\\[2mm]
0&0&2\, \left( L_{{E}}-R_{{E}} \right) ^{3}
\end{smallmatrix}
\right).
\end{multline*}
is the LU-decomposition of the system \eqref{eq:intro, first}.
The goal of this paper is to obtain a transparent expression
for the LU-decomposition of the following more general system
\begin{equation}\label{eq:B}
E^np(H+n)=p(H-n)E^n  \quad  \text{for all} \,\, n\in\N,
\end{equation}
for some $H\in\mathcal{A}$.

\subsection{Relevance of the system (\ref{eq:B})}\label{subsec.intr-motiv}
The interest on the system \eqref{eq:B} comes from the fact that their
solution set is closely related to
invariant spaces under group actions.
This is clear in the particular case in which $H=0$ and $E$ is the identity of
$\mathcal{A}$, since the solutions of \eqref{eq:B} are the even polynomials.
It is not difficult to see that in other more general situations
there exists a group $G_E$ associated to $E$,
acting on $\mathcal{M}\otimes_\k\k[t]$,
such that
\[
\{p\in\mathcal{M}\otimes_\k\k[t]:\,E^np(n)=p(-n)E^n  \text{ for all }  n\in\N\}=
\big(\mathcal{M}\otimes_\k\k[t]\big)^{G_E}.
\]

A second, and more relevant, example of the relationship between
the system \eqref{eq:B} and invariants of groups
appears in a problem from representation theory of Lie groups.
Let $G_o$ be a connected, noncompact, real semisimple Lie group with finite
center, let $\lieg$ be the complexification of the Lie algebra of $G_o$
and let $K_o$ denote a maximal
compact subgroup of $G_o$.
By the fundamental work of Harish-Chandra it is known that many deep
questions concerning the infinite dimensional
representation theory of $G_o$ reduce to  questions about the structure
of the algebra $U(\lieg)^{K_o}$, the $K_o$-invariants
 in the universal enveloping algebra of $\lieg$.
In \cite{T}, Tirao proved that the elements of the image of $U(\lieg)^{K_o}$
by the Lepowsky homomorphism (see \cite{L}) satisfy a system of linear equations
completely analogous to system \eqref{eq:B}.
Moreover, for rank one classical Lie groups it is proved in \cite{BCT} that this
image coincides with the solution set of the system introduced by Tirao.
In order to prove this result,
the LU-decomposition obtained in this paper constitute a very useful tool.
These connections with groups invariants
 are not treated in this paper.

\subsection{Main results}

In \S\ref{sec:gausselimin} we
find a periodic sequence of elementary row operations that,
when applied to the system \eqref{eq:B},
transforms it into a triangular system
that involves the adjoint action of $E$
and some combinatoric sums
of discrete derivatives of the polynomial $p$.
The main result of this section is Theorem \ref{thm:Eqinfinity}.
As a corollary we obtain that
if $p\in{\mathcal{M}}[t]$ satisfies
\begin{equation*}
E^np(H+n)=p(H-n)E^n  \quad  \text{for } \,\, 0\le n\le\deg(p),
\end{equation*}
then $p$ satisfies \eqref{eq:B}.

Note that $M_0$ is an infinite matrix with coefficients in $\End_{\k}(\mathcal{M})$
and it is closely related to the matrix
\[
\tilde M_0=\left(
\begin{smallmatrix}
x   - 1   \;&\; x    + 1    \;&\; x    -  1   \;&\;  x   +   1   &\;\; \cdot  \\[3mm]
x^2 - 1   \;&\; 2x^2 + 2    \;&\; 4x^2 - 4    \;&\; 8x^2 +  8    &\;\; \cdot  \\[3mm]
x^3 - 1   \;&\; 3x^3 + 3    \;&\; 9x^3 - 9    \;&\;27x^3 + 27    &\;\; \cdot  \\[3mm]
.           & .         & .         & .
\end{smallmatrix}
\right)\in\k[x].
\]
Taking into account the relationship between $M_0$ and $\tilde M_0$,
the results obtained in Theorem \ref{thm:Eqinfinity}
can be translated in terms of
polynomials in one variable with coefficients in $\k$.
The combinatorial aspects of Theorem \ref{thm:Eqinfinity}
are transformed into very simple matrix identities,
where the entries of the matrices involved are a special class of Jacobi polynomials.
In turns, these matrix identities yield biorthogonality relations between
the ultraspherical polynomials.
It has recently appeared in the literature some other matrix identities (in particular LU decompositions)
involving Jacobi polynomials that translate
sophisticated polynomials identities into very simple matrix identities
(see for instance \cite{KO}).

These results are described in what follows.

\medskip

Let $P^{\alpha,\beta}_n$ denote the Jacobi polynomial of degree $n$
associated to the numbers $\alpha$ and $\beta$.
Usually $\alpha$ and $\beta$ are complex numbers but, since we are working in the context of
an arbitrary field of characteristic zero, they are assumed to be rational.
Recall that the polynomials corresponding to parameters $\alpha=\beta$
are known as the \emph{ultraspherical} or \emph{Gegenbauer's} polynomials.
Let
\[
p_n^{\alpha,\beta}(x)=(x-1)^nP^{\alpha,\beta}_n\left(\tfrac{x+1}{x-1}\right).
\]
It is clear that $p_n^{\alpha,\beta}(x)$ is again a polynomial
and it can be expressed in terms of the hypergeometric
function of Gauss as follows (see (4.22.1) in \cite{Sz}),
\[
p_n^{\alpha,\beta}(x)=
\tfrac{(n+\alpha+\beta+1)\dots(2n+\alpha+\beta)}{n!}\;
{}_2F_1\left(-n,-n-\alpha;-2n-\alpha-\beta;x-1\right).
\]
Our results in terms of $\tilde M_0$ are the following:
\begin{enumerate}[(1)]
  \item
The LU-decomposition of $\tilde M_0$ is
\[
\tilde M_0=\tilde L_0 \tilde U_0
\left(
\begin{smallmatrix}
1     &         &         &         &         &         \\
      & 1       &         &         &         &          \\
      &         & 2       &         &         &           \\
      &         &         &6        &         &           \\
      &         &         &         & 24      &             \\
      &         &         &         &         & \cdot
\end{smallmatrix}
\right)
{\tiny
\left(
\begin{array}{rrrrrr}
1       &-1      & 1    &-1    & 1      & \cdot      \\[1mm]
        & 1      &-3    & 7    &-15     & \cdot      \\[1mm]
        &        & 1    &-6    & 25     & \cdot      \\[1mm]
        &        &      & 1    &-10     & \cdot      \\[1mm]
        &        &      &      & 1      & \cdot      \\[1mm]
.       & .      & .    & .    & .      \\
\end{array}
\right),
}
\]
where
\[
\begin{array}{rclr}
\displaystyle (\tilde L_0)_{ij}&=&\displaystyle (-1)^{i-j}\;p_{i-j}^{-i,-i}(x), &i\ge j; \\[3mm]
\displaystyle (\tilde U_0)_{ij}&=&\displaystyle (-1)^{i-j}\;\frac{j}{i}\;(x-1)^{2i-j}\;p_{j-i}^{-j,-1}(x), &i\le j;\\[2mm]
\end{array}
\]
and the explicit numerical matrices are respectively formed by the factorial and
(up to the minus signs) Stirling numbers of the second kind.

\item The inverse of the matrix $\tilde L_0$ is
\[
(\tilde L_0^{-1})_{ij}=(-1)^{i-j}\;\frac{j}{i}\;p_{i-j}^{j,j}(x),\quad i\ge j.
\]
This yields the following ``discrete orthogonality'' relationship that
involves once many of the ultraspherical  Jacobi polynomials
with integer parameters
\[
\left(
\begin{smallmatrix}
\frac11 P_0^{1,1} &                  &                 &                  & \cdot   \\[1mm]
\frac12 P_1^{1,1} &\frac22 P_0^{2,2} &                 &                  & \cdot   \\[1mm]
\frac13 P_2^{1,1} &\frac23 P_1^{2,2} &\frac33 P_0^{3,3}&                  & \cdot   \\[1mm]
\frac14 P_3^{1,1} &\frac24 P_2^{2,2} &\frac34 P_1^{3,3}&\frac11 P_0^{4,4} & \cdot   \\[1mm]
.                 & .                & .               & .
\end{smallmatrix}
\right)
\left(
\begin{smallmatrix}
P_0^{-1,-1} &             &            &             & \cdot   \\[1.23mm]
P_1^{-2,-2} & P_0^{-2,-2} &            &             & \cdot   \\[1.23mm]
P_2^{-3,-3} & P_1^{-3,-3} & P_0^{-3,-3}&             & \cdot   \\[1.23mm]
P_3^{-4,-4} & P_2^{-4,-4} & P_1^{-4,-4}& P_0^{-4,-4} & \cdot   \\[1.23mm]
.           & .           & .          & .
\end{smallmatrix}
\right)=
\left(
\begin{smallmatrix}
1\:\: &    \:\:&  \:\:   &  \:\: & \cdot    \\[2.20mm]
 \:\: & 1  \:\:&  \:\:   &  \:\: &  \cdot   \\[2.20mm]
 \:\: &    \:\:& 1\:\:   &  \:\: &  \cdot   \\[2.20mm]
 \:\: &    \:\:&  \:\:   & 1\:\: &  \cdot   \\[2.20mm]
.\:\: & .  \:\:& .\:\:   & .
\end{smallmatrix}
\right).
\]

\item
We mentioned above that we performed to the system \eqref{eq:B}
a gaussian elimination process in which we periodically apply the same sequence of
elementary row operations.
This periodicity is translated in the following terms.
Given an infinite matrix
\[
T=\left(
\begin{smallmatrix}
a_{11} & a_{12} & a_{13} & \cdot\;\; \\[1.5mm]
a_{21} & a_{22} & a_{23} & \cdot\;\; \\[1.5mm]
a_{31} & a_{32} & a_{33} & \cdot\;\; \\
. & . & . &
\end{smallmatrix}
\right),
\]
let $s(T)$
be the \emph{shifted matrix} of $T$, that is $s(T)$ is
the diagonal blocked matrix
$
\left(
  \begin{smallmatrix}
1     &    \\
          & T
\end{smallmatrix}
\right)
$
formed by the identity
matrix of size $1\times 1$ and $T$.
We say that an infinite matrix $A$
admits a \emph{periodic gaussian elimination process} if there exist
an lower triangular matrix with 1's in the diagonal $T_0$ such that
the sequence of matrices
\[
A,\quad T_0A,\quad s(T_0)T_0A,\quad s^2(T_0)s(T_0)T_0A, \quad\cdots \]
converges\footnote{The concept of convergence is used in the discrete sense, that is
a sequence of matrices $B_k$ converges to $B$ if for
every $i$, $j$ there exists $k_0$ such that
$(B_k)_{ij}=B_{ij}$ for all $k\ge k_0$.}
 to an  upper triangular matrix.
In this case $T_0$ is called a \emph{fundamental sequence}
of the Gaussian elimination process of $A$.
It is not difficult to see that an infinite matrix $A$
admits a periodic gaussian elimination process if and only if
$A$ admits an LU-decomposition.
In fact, if $A=LU$ then the fundamental period is $T_0=s(L)L^{-1}$ and,
conversely, if $T_0$ is the fundamental period of $A$ then
\[L^{-1}=\cdots s^2(T_0)s(T_0)T_0\quad\text{ and }\quad
L=T_0^{-1}s(T_0^{-1})s^2(T_0^{-1})\cdots.\]
It is worth mentioning at this point that
the Vandermonde matrix $V_{ij}=i^{j-1}$ admits a
periodic gaussian elimination process and its fundamental sequence is
$\left(
\begin{smallmatrix}
\;\;\,1      &      &      &      & \cdot      \\
-1     &\;\;\,1    &      &      & \cdot      \\
       &-1    &\;\;\,1    &&       \cdot      \\
       &       &-1    &\;\;\,1    & \cdot      \\
\;\;\,.    & \;\;\,.      & \;\;\,.    &
\end{smallmatrix}
\right)$
and that the recursive definition of the Pascal numbers is equivalent to
the identity
\[
\left(
\begin{smallmatrix}
1\;     &        &        &        & \cdot\;\;      \\[1mm]
1\;     & 1\;    &        &        & \cdot\;\;      \\[1mm]
1\;     & 2\;    & 1\;    &        & \cdot\;\;      \\[1mm]
1\;     & 3\;    & 3\;    & 1\;    & \cdot\;\;      \\[1mm]
.\;     & .\;    & .\;    & .\;  \\
\end{smallmatrix}
\right)=\cdots s^2(T_0)s(T_0)T_0,\quad\text{with }\;
T_0=
\left(
\begin{smallmatrix}
1      &     &     &      & \cdot      \\
1      &1    &     &      & \cdot      \\
       &1    &1    &      & \cdot      \\
       &     &1    &1    & \cdot      \\
.      &.    &.    &.    &
\end{smallmatrix}
\right).
\]

It turns out that $\tilde M_0$ admits a
periodic gaussian elimination process and we show that its fundamental sequence is
\[
T_0=
\left(
\begin{smallmatrix}
1       &        &         &          & \cdot      \\
-x-1    &1       &         &          & \cdot      \\
x       &-x-1    &1        &          & \cdot      \\
        &x       &-x-1    &1          & \cdot      \\
.       & .      & .      & .
\end{smallmatrix}
\right).
\]
It is not difficult to see that the inverse of $T_0$ is the matrix formed by the cyclotomic polynomials
\[
T_0^{-1}=
\left(
\begin{smallmatrix}
1          &        &         &          & \cdot      \\[1mm]
1+x        &1       &         &          & \cdot      \\[1mm]
1+x+x^2    &1+x     &1        &          & \cdot      \\[1mm]
1+x+x^2+x^3&1+x+x^2 &1+x      &1         & \cdot      \\[1mm]
.          & .      & .       & .
\end{smallmatrix}
\right),
\]
and therefore the ultraspherical polynomials  can expressed as
\begin{multline*}
\left(
\begin{smallmatrix}
\frac11 p_0^{1,1} &                  &                 &                  & \cdot   \\[1mm]
\frac12 p_1^{1,1} &\frac22 p_0^{2,2} &                 &                  & \cdot   \\[1mm]
\frac13 p_2^{1,1} &\frac23 p_1^{2,2} &\frac33 p_0^{3,3}&                  & \cdot   \\[1mm]
\frac14 p_3^{1,1} &\frac24 p_2^{2,2} &\frac34 p_1^{3,3}&\frac11 p_0^{4,4} & \cdot   \\[1mm]
.                 & .                & .               & .
\end{smallmatrix}
\right)\\
=
\cdots
\left(
\begin{smallmatrix}
1       &        &         &          & \cdot      \\[1.5mm]
        &\;\;1   &         &          & \cdot      \\[1.5mm]
        &        &1        &          & \cdot      \\[1.5mm]
        &        &-x-1     &1         & \cdot      \\
.       & .      & .       & .
\end{smallmatrix}
\right)
\left(
\begin{smallmatrix}
1       &        &         &          & \cdot      \\[1.5mm]
        &1       &         &          & \cdot      \\[1.5mm]
        &-x-1    &1        &          & \cdot      \\[1.5mm]
        &x       &-x-1     &1         & \cdot      \\
.       & .      & .       & .
\end{smallmatrix}
\right)
\left(
\begin{smallmatrix}
1       &        &         &          & \cdot      \\[1.5mm]
-x-1    &1       &         &          & \cdot      \\[1.5mm]
x       &-x-1    &1        &          & \cdot      \\[1.5mm]
        &x       &-x-1     &1         & \cdot      \\
.       & .      & .       & .
\end{smallmatrix}
\right)
\end{multline*}
and
\begin{multline*}
\left(
\begin{smallmatrix}
p_0^{-1,-1} &             &            &             & \cdot   \\[1.23mm]
p_1^{-2,-2} & p_0^{-2,-2} &            &             & \cdot   \\[1.23mm]
p_2^{-3,-3} & p_1^{-3,-3} & p_0^{-3,-3}&             & \cdot   \\[1.23mm]
p_3^{-4,-4} & p_2^{-4,-4} & p_1^{-4,-4}& p_0^{-4,-4} & \cdot   \\[1.23mm]
.           & .           & .          & .
\end{smallmatrix}
\right) \\
=
\left(
\begin{smallmatrix}
1          &        &         &          & \cdot      \\[1mm]
1+x        &1       &         &          & \cdot      \\[1mm]
1+x+x^2    &1+x     &1        &          & \cdot      \\[1mm]
1+x+x^2+x^3&1+x+x^2 &1+x      &1         & \cdot      \\[1mm]
.          & .      & .       & .
\end{smallmatrix}
\right)
\left(
\begin{smallmatrix}
1\;\;      &        &         &          & \cdot      \\[1.5mm]
           &1       &         &          & \cdot      \\[1mm]
           &1+x     &1        &          & \cdot      \\[1mm]
           &1+x+x^2 &1+x      &1         & \cdot      \\[1mm]
.          & .      & .       & .
\end{smallmatrix}
\right)
\left(
\begin{smallmatrix}
1\;\;      &        &         &          & \cdot      \\[1.5mm]
           &1\;\;   &         &          & \cdot      \\[1.5mm]
           &        &1        &          & \cdot      \\[1mm]
           &        &1+x      &1         & \cdot      \\[1mm]
.          & .      & .       & .
\end{smallmatrix}
\right)
\cdots
\end{multline*}
We notice that in these two identities the polynomials involved are
$p_n^{\alpha,\alpha}$ instead of $P_n^{\alpha,\alpha}$.
\end{enumerate}

\section{The gaussian elimination for the system \eqref{eq:B}}\label{sec:gausselimin}

Let $\k$ be a field of characteristic zero,
let $\mathcal{A}$ be an associative, not necessarily commutative, $\k$-algebra with unit and
let $\mathcal{M}$ be a unital $\mathcal{A}$-bimodule.
Given an element $r\in\mathcal{A}$, let $L_r,R_r\in\End_{\k}(\mathcal{M})$
respectively denote the left and right actions
by $r$, and let $\ad_r=L_r-R_r$ be the adjoint action of $r$ in $\mathcal{M}$.

Recall that $\mathcal{M}[t]=\mathcal{M}\otimes_\k\k[t]$, and that given an element
$r\in\mathcal{A}$ one has the evaluation map
$\mathcal{M}[t]\to\mathcal{M}$ defined using the right action of $\mathcal{A}$ on $\mathcal{M}$
by $a\otimes p(t)\mapsto ap(r)$.
If $r\in\mathcal{A}$ and $p\in\mathcal{M}[t]$ then the evaluation of $p$ in $r$
is given by $p(r)=p_0+p_1 r+\dots+p_n r^n$.

Given two arbitrary elements $E$ and $H$ in $\mathcal{A}$ we are
interested in the set of polynomials $p\in{\mathcal{M}}[t]$ that
satisfy the following system of linear equations \eqref{eq:B},
that is polynomials $p\in{\mathcal{M}}[t]$ satisfying
\begin{equation*}
E^np(H+n)=p(H-n)E^n\quad \text{for all }n\in\N.
\end{equation*}
In this section, we assume that $E$ is an eigenvector of $\ad_H$ in $\mathcal{A}$
with eigenvalue $c$ and we shall perform a gaussian elimination process
to obtain, in Theorem \ref{thm:Eqinfinity},
a triangular linear system in $\mathcal{M}[t]$ equivalent to the system \eqref{eq:B}.

\subsection{A first look at the system}

We first discuss the linear system \eqref{eq:B} for polynomials of degree zero and one.

\vspace{1mm}

\noindent \emph{Degree zero:} Let $p=p_0\in{\mathcal{M}}[t]$ be a
constant polynomial. Observe that $p$ satisfies the first equation
of the system if and only if $p_0$ commutes with $E$, and therefore
$p$ satisfies all the other equations. Thus $p$ satisfies the system
\eqref{eq:B} if and only if $p_0\in\ker\ad_E$.

\vspace{1mm}

\noindent
\emph{Degree one:}  A linear polynomial
$p=p_0+p_1t\in{\mathcal{M}}[t]$ satisfies the system \eqref{eq:B} if and only if
the vector $(p_0,p_1)$ is a solution of the following system,
\begin{eqnarray}
  E  p_0\;   + E  p_1(H\!+\!1)\;\,     &= & p_{0}E\;\;  + p_{1}(H\!-\!1) E\;\, \label{1} \\
  E^2  p_0 + E^2  p_1(H\!+\!2) &=& p_{0}E^2   + p_{1}(H\!-\!2) E^2 \label{2}  \\
  E^3  p_0 + E^3  p_1(H\!+\!3)  &=& p_{0}E^3   + p_{1}(H\!-\!3) E^3 \label{3} \\
 \vdots\hspace{1cm} && \hspace{1cm}\vdots\nonumber
\end{eqnarray}
It is not difficult to see that equation \eqref{3} can be obtained
from the first two as follows,
\[
2E\times(\ref{2})-E^2\times(\ref{1})-E\times(\ref{1})\times E
-(\ref{1})\times E^2+(\ref{2})\times E.
\]
In fact, we shall see in Corollary \ref{coro:first N} that for all $n\ge3$,
the $n\text{th}$ equation can be obtained from equations \eqref{1} and \eqref{2}. Then,
in this case, the system \eqref{eq:B} is equivalent to the $2\times2$
system formed by equations \eqref{1} and \eqref{2}. This $2\times2$
system is represented by the matrix
$$M=\begin{pmatrix}
\scriptstyle L_E-R_E & \scriptstyle L_ER_{H+1}-R_ER_{H-1} \\[2mm]
\scriptstyle L_E^2-R_E^2 & \scriptstyle L_E^2R_{H+2}-R_E^2R_{H-2}
\end{pmatrix},$$
with coefficients in $\End_{\k}(\mathcal{M})$. Now, since
\[
\begin{pmatrix}
\scriptstyle 1 & \scriptstyle 0 \\[2mm]
\scriptstyle -(L_E+R_E) & \scriptstyle 1
\end{pmatrix}
\begin{pmatrix}
\scriptstyle L_E-R_E & \scriptstyle L_ER_{H+1}-R_ER_{H-1} \\[2mm]
\scriptstyle L_E^2-R_E^2 & \scriptstyle L_E^2R_{H+2}-R_E^2R_{H-2}
\end{pmatrix}
=
\begin{pmatrix}
\scriptstyle  L_E-R_E & \scriptstyle L_ER_{H+1}-R_ER_{H-1} \\[2mm]
\scriptstyle 0& \scriptstyle (L_E-R_E)^2
\end{pmatrix},
\]
we obtain that equations \eqref{1} and \eqref{2},
are equivalent to the triangular system
$U\left(
\begin{smallmatrix}
 b_0 \\
 b_1
\end{smallmatrix}\right)=0,
$
with
\[
U=
\left(
\begin{smallmatrix}
 L_E-R_E & L_ER_{H+1}-R_ER_{H-1} \\[2mm]
0& (L_E-R_E)^2
\end{smallmatrix}\right)=
\left(
\begin{smallmatrix}
 L_E-R_E & L_E+R_E \\[2mm]
0& (L_E-R_E)^2
\end{smallmatrix}\right)
\left(
\begin{smallmatrix}
 1 & R_{H} \\[3mm]
0& 1
\end{smallmatrix}\right).
\]

In Theorem \ref{thm:Eqinfinity} we shall extend this triangularization process
to polynomials of arbitrary degree. In particular, we shall obtain in
Corollary \ref{coro:first N} that for polynomials of degree $N$ the whole
system \eqref{eq:B} is equivalent to the first $N+1$ equations of the
system \eqref{eq:B}.

\subsection{Some additional notation}\label{subsec:notation}
For $h\in\k$, $h\ne0$, we define the $h$-discrete derivative
$\partial_h:\mathcal{M}[t]\to\mathcal{M}[t]$ by
\[
\partial_hp(t)=\frac{p(t)-p(t-h)}{h}, \quad\text{ for }p\in\mathcal{M}[t].
\]
It is clear that
\[
\partial_h^np(t)=\frac1{h^n}\sum^n_{j=0} (-1)^j  {\binom n j}p(t-jh),
\]
and, if $p=p_mt^m+\cdots+p_0$, one has
$$
\partial_h^np(t)=\begin{cases}
0,&\text{if $n>m$}\\ m! p_m ,&\text{if $n=m$}.
\end{cases}
$$

If $(t)_n$ is the Pochhammer polynomial of degree $n$, that is
\[
(t)_{0}=1 \quad \text{ and }\quad (t)_n=t(t+1)\dots(t+n-1) \quad \text{ for } n\ge 1,
\]
it is easy to see that $\partial_h (t/h)_n=n(t/h)_{n-1}$ and thus
\begin{align}
\partial_h^r (t/h)_n&=r!\binom{n}{r}(t/h)_{n-r} \notag \\ \label{eq:Pochhammer}
&=(n-r+1)_{r}(t/h)_{n-r}\quad\text{ for } r\le n.
\end{align}

Given $r\in\mathcal{A}$ we extend to $\mathcal{M}[t]$
the adjoint action of $r$ in $\mathcal{M}$ by
\[
\ad_r(p_0+p_1 t+\dots+p_n t^n)=\ad_r(p_0)+\ad_r(p_1) t+\dots+\ad_r(p_n) t^n.
\]
Observe that $\ad_r$ commute with the discrete derivative
$\partial_h$ for all $h\in\k$.
In addition, if $E$, $H\in\mathcal{A}$ are elements that satisfy
$\ad_H(E)=cE$, $c\in\k$, then it is straightforward to prove that
\begin{align}
\label{eq:id1}
\ad_E\big(p(t+H)\big)&=\ad_E(p)(t+H)-c\,\partial_{c}p(t+H)E \\
\label{eq:id2}
E^np(t+H)&=\sum_{j=0}^n\binom nj\ad_E^{n-j}(p)(t-cj+H)E^j
\text{ for all $n\in\N$} \\
\label{eq:id3}
E^nH^m&=(H-nc)^mE^n\text{ for all $m,n\in\N$}.
\end{align}

\subsection{The gaussian elimination}\label{subsec:def eq}
For $k\ge0$ and $n\ge k$ we introduce the following equation,
\begin{multline*}
\eq_n^k:\sum_{i=0}^k\sum_{j=0}^{n-k}(-1)^i \tbinom{k}{i}\tbinom{n-k}{j}\,
\ad_E^{n-(j+i)}\partial_1^k\partial_{c-1}^i\;p(H+n-cj-i)E^{j+i} \\
=\sum_{i=0}^k(-1)^i\tbinom{k}{i}\,
\ad_E^{k-i}\partial_1^k\partial_{c-1}^i\;p(H-n+2k-i)E^{n-(k-i)}.
\end{multline*}

We now collect some properties of these equations.
\begin{enumerate}[(1)]
\item The equation $\eq_n^0$ is,\,
$\textstyle
\sum_{j=0}^{n}\tbinom{n}{j}\,\ad_E^{n-j}p(H+n-cj)E^{j} =
p(H-n)E^{n}.$ Hence, in view of the identity \eqref{eq:id2}, the
system $\{\eq_n^0:n\in \mathbb{N}\}$ is the original system \eqref{eq:B}.

\item On the other hand, the equation
$\eq_n^n$ is
\begin{multline*}
\sum_{i=0}^n(-1)^i \tbinom{n}{i}\,
\ad_E^{n-i}\partial_1^n\partial_{c-1}^i\;p(H+n-i)E^{i} \\
=\sum_{i=0}^n(-1)^i\tbinom{n}{i}\,
\ad_E^{n-i}\partial_1^n\partial_{c-1}^i\;p(H-n+2n-i)E^{i},
\end{multline*}
which is trivial for all $n\in \mathbb{N}$.

\item The fundamental property of $\eq_n^k$, for $k\ge0$ and $n\ge k$,
is that the polynomial $p$ appears under a derivative of order
grater than or equal to $k$.
This implies that for all $n\ge k$ the equation
$\eq_n^k$ involves only the coefficients $p_j$
with $j\ge k$.
\end{enumerate}

For $k\ge0$, let $\Eq^k$ be the system
\[
\Eq^k=\{\eq_1^0,\eq_2^1,\dots,\eq_{k}^{k-1}\}\cup\{\eq_n^{k}:n>k\}.
\]
As we have just observed, $\Eq^0$ is the original system \eqref{eq:B}.
The following picture may help to keep these definitions in mind.

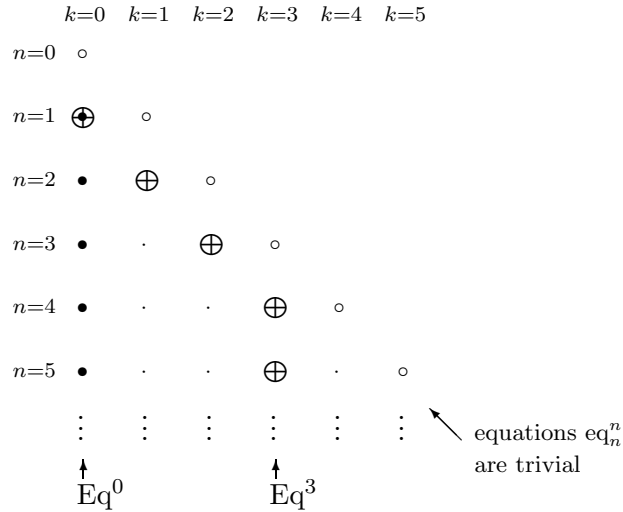
\begin{center}
\setlength{\unitlength}{24pt}
\begin{picture}(9,8)(0,1)
  \put(0.8,8.6){$\scriptstyle k=0$}
  \put(1.8,8.6){$\scriptstyle k=1$}
  \put(2.8,8.6){$\scriptstyle k=2$}
  \put(3.8,8.6){$\scriptstyle k=3$}
  \put(4.8,8.6){$\scriptstyle k=4$}
  \put(5.8,8.6){$\scriptstyle k=5$}
  \put(0,8){$\scriptstyle n=0$}
  \put(0,7){$\scriptstyle n=1$}
  \put(0,6){$\scriptstyle n=2$}
  \put(0,5){$\scriptstyle n=3$}
  \put(0,4){$\scriptstyle n=4$}
  \put(0,3){$\scriptstyle n=5$}
  \put(1.1,1.4){\vector(0,1){.3}}
  \put(1,1){$ \Eq^0$}
  \put(4.1,1.4){\vector(0,1){.3}}
  \put(4,1){$ \Eq^3$}
  \put(7.0,2.0){\vector(-1,1){.5}}
  \put(7.2,1.5){\Small  are trivial}
  \put(7.2,2){\Small  equations $\eq_n^n$}
  \put(1,8){$\scriptstyle \circ$}
  \put(1,7){$\scriptstyle \bullet$}
  \put(.90,7){$\scriptstyle \bigoplus$}
  \put(1,6){$\scriptstyle \bullet$}
  \put(1,5){$\scriptstyle \bullet$}
  \put(1,4){$\scriptstyle \bullet$}
  \put(1,3){$\scriptstyle \bullet$}
  \put(1,2){$\scriptstyle \vdots$}
  \put(2,7){$\scriptstyle \circ$}
  \put(1.9,6){$\scriptstyle \bigoplus$}
  \put(2,5){$\scriptstyle \cdot$}
  \put(2,4){$\scriptstyle \cdot$}
  \put(2,3){$\scriptstyle \cdot$}
  \put(2,2){$\scriptstyle \vdots$}
  \put(3,6){$\scriptstyle \circ$}
  \put(2.9,5){$\scriptstyle \bigoplus$}
  \put(3,4){$\scriptstyle \cdot$}
  \put(3,3){$\scriptstyle \cdot$}
  \put(3,2){$\scriptstyle \vdots$}
  \put(4,5){$\scriptstyle \circ$}
  \put(3.9,4){$\scriptstyle \bigoplus$}
  \put(3.9,3){$\scriptstyle \bigoplus$}
  \put(4,2){$\scriptstyle \vdots$}
  \put(5,4){$\scriptstyle \circ$}
  \put(5,3){$\scriptstyle \cdot$}
  \put(5,2){$\scriptstyle \vdots$}
  \put(6,3){$\scriptstyle \circ$}
  \put(6,2){$\scriptstyle \vdots$}
\end{picture}
\end{center}

Finally, let $\Eq^\infty$ be the system
\[
\Eq^\infty=\{\eq_{n+1}^n:n\ge0\}.
\]
According to item (3) above, $\Eq^\infty$ is an upper triangular
system, in the sense that if $p=\sum_{j\ge0} p_j\,t^j$ then
$\eq_{n+1}^n$ only involves the unknowns $p_j$ with $j\ge n$. The
equations of this system are the following,
\begin{multline*}
\eq_{n+1}^n:\sum_{i=0}^{n}(-1)^i \tbinom{n}{i}\,
\ad_E^{n+1-i}\partial_1^{n}\partial_{c-1}^i\;p(H+n+1-i)E^{i} \\
+\sum_{i=0}^{n}(-1)^i \tbinom{n}{i}\,
\ad_E^{n-i}\partial_1^{n}\partial_{c-1}^i\;p(H+n+1-c-i)E^{i+1} \\
=\sum_{i=0}^{n}(-1)^i\tbinom{n}{i}\,
\ad_E^{n-i}\partial_1^{n}\partial_{c-1}^i\;p(H+n-1-i)E^{i+1}.
\end{multline*}
\begin{remark}\label{rmk:degree}
If $\deg(p)=N$ then the only non zero equations of $\Eq^\infty$ are
$\eq_{n+1}^n$ for $n=0,\dots, N$. In this case, the last two non
zero equations are, \vspace{1mm}
\[
\begin{array}{rrl}
\eq_{N\!+\!1}^N\!\!:& \ad_E^{N+1}(p_N)=&\!\!\!0, \\[2mm]
\eq_{N}^{N\!-\!1}\!\!:&\!\!\!\!
\ad_E^{N}(p_{N\!-\!1})+N\ad_E^{N}(p_N)(H\!+\!\tfrac{N\!+\!1}2)
+N(N\!+\!1\!-\!Nc)\ad_E^{N\!-\!1}(p_N)E=&\!\!\!0 .
\end{array}
\]
Indeed, the equation $\eq_{N+1}^N$ is
\[
\ad_E^{N+1}\partial_1^{N}\,p(H\!+\!N\!+\!1)
+\ad_E^{N}\partial_1^{N}\,p(H\!+\!N\!+\!1\!-\!c)E\\
=\ad_E^{N}\partial_1^{N}\,p(H\!+\!N\!-\!1)E,
\]
however, since $\partial_1^{N}\,p=N!\,p_N$ the equation $\eq_{N+1}^N$
becomes $\ad_E^{N+1}(p_N)=0$. For $\eq_{N}^{N-1}$ we use that
$\partial_1^{N-1}p=(N\!-\!1)!\,(Np_N\,t-\tbinom{N}{2}p_N+p_{N-1}),$
and proceed in a similar way.
\end{remark}

\vspace{1mm}

The main theorem of this section is the following.

\begin{theorem}\label{thm:Eqinfinity}
Assume that $[H,E]=cE$.
For all $k\ge0$, the first $N$ equations of the system $\Eq^k$
are equivalent to the first $N$ equations of the system $\Eq^0$.
In particular this is true for the upper triangular system $\Eq^\infty$.
\end{theorem}

\begin{proof}
The theorem easily follows by induction once we prove
that the equations $\eq^{k}_n$ satisfy
the following recursion formula
\begin{equation}\label{prop:recursion}
\eq^{k}_n=
\eq^{k-1}_n-2\eq^{k-1}_{n-1}E+\eq^{k-1}_{n-2}E^2
-\ad_E(\eq^{k-1}_{n-1})+\ad_E(\eq^{k-1}_{n-2})E
\end{equation}
for $n> k> 0$.

The proof of this recursion formula is very technical and will be done
at the end of this section.
We point out that this formula shows that, for $n>k$, the
equation $\eq^k_n$ of $\Eq^k$ is equal to the equation $\eq^{k-1}_n$
of $\Eq^{k-1}$ plus a linear combination of the two previous
equations of $\Eq^{k-1}$.
\end{proof}

\begin{corollary}\label{coro:first N}
Assume that $[H,E]=cE$.
If $p \in {\mathcal{M}}[t]$
satisfies
\[E^np(H+n)=p(H-n)E^n\;\text{ for $n=1,\dots,\deg(p)$},\]
then $E^np(H+n)=p(H-n)E^n$ holds for all
$n\ge0$.
\end{corollary}
\begin{proof}
Let $N=\deg(p)$. The assumption on $p$ says that it satisfies the first $N$
equations of the system $\Eq^{0}$, then Theorem \ref{thm:Eqinfinity}
implies that $p$ satisfies the first $N$ equations of
$\Eq^\infty$. On the other hand, we showed in Remark \ref{rmk:degree}
that $p$ automatically satisfies all the remaining equations of
$\Eq^\infty$.
Therefore $p$ satisfies all the equations of $\Eq^0$.
\end{proof}

In the following corollary we describe the particular cases in which
$c=0$ and $c=1$, in this cases the triangularized system became much
simpler.

\begin{corollary}\label{coro:c=0,c=1}
If $\ad_H(E)=0$ then the equations of the system $\Eq^\infty$  are
\begin{align*}
\eq_{n+1}^n:\sum_{i=0}^{n} \Big[\tbinom{n}{i}\,
\ad_E^{n\!+\!1\!-\!i}\partial_1^{n\!+\!i}\;p(H\!+\!n\!+\!1)&E^{i}
+\tbinom{n}{i}\ad_E^{n\!-\!i}\partial_1^{n\!+\!i}\;p(H\!+\!n\!+\!1)E^{i\!+\!1} \Big] \\
=&\sum_{i=0}^{n}\tbinom{n}{i}\,
\ad_E^{n\!-\!i}\partial_1^{n\!+\!i}\;p(H\!+\!n\!-\!1)E^{i\!+\!1}.
\end{align*}
If $\ad_H(E)=E$ then the equations of the system $\Eq^\infty$  are
\[
\eq_{n+1}^n:\ad_E^{n+1}\partial_1^{n}\;p(H+n+1)
+\ad_E^{n}\partial_1^{n+1}\;p(H+n)E=0.
\]
\end{corollary}

\begin{proof}
If $E$ and $H$ commute then $c=0$, hence
\[
\partial_{c-1}p(t)=\partial_{-1}p(t)=-\partial_{1}p(t+1).
\]
Therefore $\partial_{c-1}^{i}p(t)=(-1)^{i}\partial_{1}^ip(t+i)$ and
the formula follows.  On the other hand, if $\ad_H(E)=E$ then $c=1$,
and hence $\partial_{c-1}=\partial_{0}=0$. Therefore all the terms
of $\eq_{n+1}^n$ for $i>0$ are zero and thus
\[
\eq_{n+1}^n:\ad_E^{n+1}\partial_1^{n}\;p(H+n+1)
+\ad_E^{n}\partial_1^{n}\;p(H+n)E
=\ad_E^{n}\partial_1^{n}\;p(H+n-1)E,
\]
which is equal to the equation stated in the corollary.
\end{proof}


\subsection{Proof of the Recursion Formula \ref{prop:recursion}}
We work first on the left hand sides (LHS) of the equations involved in the
Recursion Formula \ref{prop:recursion}
(see the definition of $\eq^{k}_n$ in Subsection \ref{subsec:def eq}).
In fact, we need to prove that
\[
\text{LHS}\big(
\eq^{k-1}_n-2\eq^{k-1}_{n-1}E+\eq^{k-1}_{n-2}E^2
-\ad_E(\eq^{k-1}_{n-1})+\ad_E(\eq^{k-1}_{n-2})E\big)
=\text{LHS}\big(\eq^{k}_n\big).
\]
Since the sums in the definition of $\eq^{k}_n$ runs over all
possible values of $i$ and $j$ for which the
binomial numbers are not zero we omit them.
\begin{align*}
\text{LHS}\big(\eq^{k-1}_n\big)           &:\\
  (-1)^i \tbinom{k-1}{i}\tbinom{n-k+1}{j}   &\,
   \ad_E^{n-(j+i)}\partial_1^{k-1}\partial_{c-1}^i\;
      b(H+n-cj-i)E^{j+i} \\[2mm]
\text{LHS}\big(-2\eq^{k-1}_{n-1}E\big)         & :\\
  -(-1)^{i}\tbinom{k-1}{i}\tbinom{n-k}{j}    & \,
   \ad_E^{n-1-(j+i)}\partial_1^{k-1}\partial_{c-1}^i\;
      b(H+n-1-cj-i)E^{j+i+1} \\
  -(-1)^{i}\tbinom{k-1}{i}\tbinom{n-k}{j}    & \,
   \ad_E^{n-1-(j+i)}\partial_1^{k-1}\partial_{c-1}^i\;
      b(H+n-1-cj-i)E^{j+i+1} \\[2mm]
\text{LHS}\big(+\eq^{k-1}_{n-2}E^2\big)       & :\\
   +(-1)^i\tbinom{k-1}{i}\tbinom{n-1-k}{j}    & \,
      \ad_E^{n-2-(j+i)}\partial_1^{k-1}\partial_{c-1}^i\;
        b(H+n-2-cj-i)E^{j+i+2} \\[2mm]
\text{LHS}\big(-\ad_E(\eq^{k-1}_{n-1})\big)          & :\\
-(-1)^{i}\tbinom{k-1}{i} \tbinom{n-k}{j}            & \,
   \ad_E^{n-(j+i)}\partial_1^{k-1}\partial_{c-1}^i\;
     b(H+n-1-cj-i)E^{j+i} \\
       -(-1)^{i}\tbinom{k-1}{i}\tbinom{n-k}{j}      & \,
         \ad_E^{n-1-(j+i)}\partial_1^{k-1}\partial_{c-1}^i\;
           b(H+n-1-cj-c-i)E^{j+i+1} \\
       +(-1)^{i}\tbinom{k-1}{i}\tbinom{n-k}{j}      & \,
          \ad_E^{n-1-(j+i)}\partial_1^{k-1}\partial_{c-1}^i\;
            b(H+n-1-cj-i)E^{j+i+1} \\[2mm]
\text{LHS}\big(+\ad_E(\eq^{k-1}_{n-2})E \big)        &:\\
  +(-1)^i\tbinom{k-1}{i}\tbinom{n-1-k}{j}             &\,
   \ad_E^{n-1-(j+i)}\partial_1^{k-1}\partial_{c-1}^i\;
     b(H+n-2-cj-i)E^{j+i+1} \\
       +(-1)^{i} \tbinom{k-1}{i}\tbinom{n-1-k}{j}     &\,
         \ad_E^{n-2-(j+i)}\partial_1^{k-1}\partial_{c-1}^i\;
           b(H+n-2-cj-c-i)E^{j+i+2} \\
       -(-1)^{i} \tbinom{k-1}{i}\tbinom{n-1-k}{j}     &\,
          \ad_E^{n-2-(j+i)}\partial_1^{k-1}\partial_{c-1}^i\;
            b(H+n-2-cj-i)E^{j+i+2}
\end{align*}
We now indicate how to simplify these ten lines.

Lines 3 and 7 cancel out and so do lines 4 and 10.
Lines 6, 8 and 9, after a change of variables, become
\begin{align*}
       -(-1)^{i}\tbinom{k-1}{i}\tbinom{n-k}{j-1}      & \,
         \ad_E^{n-(j+i)}\partial_1^{k-1}\partial_{c-1}^i\;
           b(H+n-1-cj-i)E^{j+i} \\
  +(-1)^i\tbinom{k-1}{i}\tbinom{n-1-k}{j-1}             &\,
   \ad_E^{n-(j+i)}\partial_1^{k-1}\partial_{c-1}^i\;
     b(H+n-2-cj+c-i)E^{j+i} \\
       +(-1)^{i} \tbinom{k-1}{i}\tbinom{n-1-k}{j-2}     &\,
         \ad_E^{n-(j+i)}\partial_1^{k-1}\partial_{c-1}^i\;
           b(H+n-2-cj+c-i)E^{j+i},
\end{align*}
which is equal to
\begin{align*}
       -(-1)^{i}\tbinom{k-1}{i}\tbinom{n-k}{j-1}      & \,
         \ad_E^{n-(j+i)}\partial_1^{k-1}\partial_{c-1}^i\;
           b(H+n-1-cj-i)E^{j+i} \\
  +(-1)^i\tbinom{k-1}{i}\tbinom{n-k}{j-1}             &\,
   \ad_E^{n-(j+i)}\partial_1^{k-1}\partial_{c-1}^i\;
     b(H+n-2-cj+c-i)E^{j+i} \\
\end{align*}
and finally equal to
\[
 (-1)^i \tbinom{k-1}{i}\tbinom{n-k}{j-1}\,
   \ad_E^{n-(j+i)}\partial_1^{k-1}\partial_{c-1}^{i+1}\;
      b(H+n-1-c(j-1)-(i+1))E^{j+i}.
\]
Similarly, lines 1, 2 and 5 yield
\begin{align*}
  &(-1)^i \tbinom{k-1}{i}\tbinom{n-k}{j}\,
   \ad_E^{n-(j+i)}\partial_1^{k}\partial_{c-1}^i\;
      b(H+n-cj-i)E^{j+i} \\
 -&(-1)^i \tbinom{k-1}{i}\tbinom{n-k}{j-1}\,
   \ad_E^{n-(j+i)}\partial_1^{k-1}\partial_{c-1}^{i+1}\;
      b(H+n-c(j-1)-(i+1))E^{j+i}.
\end{align*}
Thus we have,
\begin{align*}
 &(-1)^i \tbinom{k-1}{i}\tbinom{n-k}{j-1}\,
   \ad_E^{n-(j+i)}\partial_1^{k-1}\partial_{c-1}^{i+1}\;
      b(H+n-1-c(j-1)-(i+1))E^{j+i}\\
  &+(-1)^i \tbinom{k-1}{i}\tbinom{n-k}{j}\,
   \ad_E^{n-(j+i)}\partial_1^{k}\partial_{c-1}^i\;
      b(H+n-cj-i)E^{j+i} \\
 -&(-1)^i \tbinom{k-1}{i}\tbinom{n-k}{j-1}\,
   \ad_E^{n-(j+i)}\partial_1^{k-1}\partial_{c-1}^{i+1}\;
      b(H+n-c(j-1)-(i+1))E^{j+i}\\[2mm]
 =-&(-1)^i \tbinom{k-1}{i}\tbinom{n-k}{j-1}\,
   \ad_E^{n-(j+i)}\partial_1^{k}\partial_{c-1}^{i+1}\;
      b(H+n-c(j-1)-(i+1))E^{j+i}\\
  &+(-1)^i \tbinom{k-1}{i}\tbinom{n-k}{j}\,
   \ad_E^{n-(j+i)}\partial_1^{k}\partial_{c-1}^i\;
      b(H+n-cj-i)E^{j+i} \\[2mm]
 =\;\;\;&(-1)^i \tbinom{k-1}{i-1}\tbinom{n-k}{j}\,
   \ad_E^{n-(j+i)}\partial_1^{k}\partial_{c-1}^{i}\;
      b(H+n-cj-i)E^{j+i}\\
  &+(-1)^i \tbinom{k-1}{i}\tbinom{n-k}{j}\,
   \ad_E^{n-(j+i)}\partial_1^{k}\partial_{c-1}^i\;
      b(H+n-cj-i)E^{j+i} \\[2mm]
 =\;\;\;&(-1)^i \tbinom{k}{i}\tbinom{n-k}{j}\,
   \ad_E^{n-(j+i)}\partial_1^{k}\partial_{c-1}^{i}\;
      b(H+n-cj-i)E^{j+i}\\[2mm]
 =\;\;\;&\text{LHS}\big(\eq^{k}_n\big).
 \end{align*}
We now work on the right hand sides (RHS) of the equations involved in the
Recursion Formula \ref{prop:recursion}.
\begin{align*}
\text{RHS}\big(\eq^{{k-1}}_n\big)      &:\\
(-1)^i\tbinom{k-1}{i}\,              &
\ad_E^{k-1-i}\partial_1^{k-1}\partial_{c-1}^i\;
b(H-n+2(k-1)-i)E^{n-(k-1-i)} \\[2mm]
\text{RHS}\big(-2\eq^{k-1}_{n-1}E\big)  & :\\
   -(-1)^i\tbinom{k-1}{i}\,             &
     \ad_E^{k-1-i}\partial_1^{k-1}\partial_{c-1}^i\;
       b(H-n+1+2(k-1)-i)E^{n-(k-1-i)} \\
   -(-1)^i\tbinom{k-1}{i}\,             &
     \ad_E^{k-1-i}\partial_1^{k-1}\partial_{c-1}^i\;
       b(H-n+1+2(k-1)-i)E^{n-(k-1-i)} \\[2mm]
\text{RHS}\big(+\eq^{k-1}_{n-2}E^2\big)       & :\\
   +(-1)^i\tbinom{k-1}{i}\,             &
     \ad_E^{k-1-i}\partial_1^{k-1}\partial_{c-1}^i\;
       b(H-n+2+2(k-1)-i)E^{n-(k-1-i)} \\[2mm]
\text{RHS}\big(-\ad_E(\eq^{k-1}_{n-1})\big)          & :\\
   -(-1)^i\tbinom{k-1}{i}\,             &
     \ad_E^{k-i}\partial_1^{k-1}\partial_{c-1}^i\;
       b(H-n+1+2(k-1)-i)E^{n-(k-i)} \\
   -(-1)^i\tbinom{k-1}{i}\,             &
     \ad_E^{k-1-i}\partial_1^{k-1}\partial_{c-1}^i\;
       b(H-n+1+2(k-1)-c-i)E^{n-(k-1-i)} \\
   +(-1)^i\tbinom{k-1}{i}\,             &
     \ad_E^{k-1-i}\partial_1^{k-1}\partial_{c-1}^i\;
       b(H-n+1+2(k-1)-i)E^{n-(k-1-i)} \\[2mm]
\text{RHS}\big(+\ad_E(\eq^{k-1}_{n-2})E \big)        &:\\
   +(-1)^i\tbinom{k-1}{i}\,             &
     \ad_E^{k-i}\partial_1^{k-1}\partial_{c-1}^i\;
       b(H-n+2+2(k-1)-i)E^{n-(k-i)} \\
   +(-1)^i\tbinom{k-1}{i}\,             &
     \ad_E^{k-1-i}\partial_1^{k-1}\partial_{c-1}^i\;
       b(H-n+2+2(k-1)-c-i)E^{n-(k-1-i)} \\
   -(-1)^i\tbinom{k-1}{i}\,             &
     \ad_E^{k-1-i}\partial_1^{k-1}\partial_{c-1}^i\;
       b(H-n+2+2(k-1)-i)E^{n-(k-1-i)}
\end{align*}
Now the sum of lines 1 and 2 is equal to
\[
-(-1)^i\tbinom{k-1}{i}\,
\ad_E^{k-1-i}\partial_1^{k}\partial_{c-1}^i\;
b(H-n-1+2k-i)E^{n-(k-1-i)},
\]
similarly the sum of lines 5 and 8 is equal to
\[
(-1)^i\tbinom{k-1}{i}\,
     \ad_E^{k-i}\partial_1^{k}\partial_{c-1}^i\;
       b(H-n+2k-i)E^{n-(k-i)},
\]
and the sum of lines 6 and 9 is equal to
\[
   (-1)^i\tbinom{k}{i}\,
     \ad_E^{k-1-i}\partial_1^{k}\partial_{c-1}^i\;
       b(H-n+2k-c-i)E^{n-(k-1-i)}.
\]
Finally, the sum of the last three lines is equal to
\begin{align*}
 &- (-1)^i\tbinom{k-1}{i}\,
     \ad_E^{k-1-i}\partial_1^{k}\partial_{c-1}^{i+1}\;
       b(H-n+2k-(i+1))E^{n-(k-1-i)}\\[2mm]
 &+ (-1)^i\tbinom{k-1}{i}\,
     \ad_E^{k-i}\partial_1^{k}\partial_{c-1}^i\;
       b(H-n+2k-i)E^{n-(k-i} \\[2mm]
=& (-1)^i\tbinom{k-1}{i-1}\,
     \ad_E^{k-i}\partial_1^{k}\partial_{c-1}^{i}\;
       b(H-n+2k-i)E^{n-(k-i)}\\[2mm]
 &+(-1)^i\tbinom{k-1}{i}\,
     \ad_E^{k-i}\partial_1^{k}\partial_{c-1}^i\;
       b(H-n+2k-i)E^{n-(k-i)} \\[2mm]
=& (-1)^i\tbinom{k}{i}\,
     \ad_E^{k-i}\partial_1^{k}\partial_{c-1}^{i}\;
       b(H-n+2k-i)E^{n-(k-i)}\\[2mm]
=&\;\; \text{RHS} \big(\eq^{k}_n\big).
\end{align*}
This finishes the proof of the Recursion Formula \ref{prop:recursion}.

\

The following two sections are devoted to express
the system \eqref{eq:B} in a matrix form and to give its LU-decomposition in a general context.
In particular we will not assume that $[H,E]=cE$.

\section{Infinite matrices \\ and the matrix representation of the system $\eqref{eq:B}$}\label{sec.MatrixRep}
In this section we recall some preliminaries on infinite
matrices, introduce the concept of \emph{periodic Gaussian
elimination process} and present the matrix representation of the system $\eqref{eq:B}$.

\subsection{Infinite matrices}
Given an associative  $\k$-algebra $\mathcal{A}$, let
$\mathcal{M}_\infty(\mathcal{A})$ denote the set of all infinite
matrices
\[
A=\begin{pmatrix}
a_{11} & a_{12} & a_{13} & \cdot\;\; \\
a_{21} & a_{22} & a_{23} & \cdot\;\; \\
a_{31} & a_{32} & a_{33} & \cdot\;\; \\
. & . & . &
\end{pmatrix}
\]
with $a_{ij}\in\mathcal{A}$. The  $n$-\emph{minor} of a matrix
$A\in\mathcal{M}_\infty(\mathcal{A})$ is the matrix corresponding to
the upper-left corner of $A$ of size $n\times n$.
We say that a sequence of matrices $B_k$ \emph{converges} to $B$ if for
every $i$, $j$ there exists $k_0$ such that
$(B_k)_{ij}=B_{ij}$ for all $k\ge k_0$.
A matrix
$A\in\mathcal{M}_\infty(\mathcal{A})$ is said \emph{row-finite}
(resp. \emph{column-finite}) if every row (resp. column) of $A$
contains only a finite number of non zero elements. Lower triangular
and upper triangular matrices are, respectively, examples of
row-finite and  column-finite matrices.

\medskip

It is clear that  $\mathcal{M}_\infty(\mathcal{A})$ is not a ring since the
multiplication of two matrices does not always exist.
Nevertheless, if $A,B\in\mathcal{M}_\infty(\mathcal{A})$
and either $A$ is row-finite or $B$ is column-finite then $AB$ do exist.
It is not difficult to prove the following proposition.

\vspace{2mm}

\begin{proposition}

\

\begin{enumerate}[(a)]
\item
A lower or upper triangular matrix
$A\in\mathcal{M}_\infty(\mathcal{A})$ is invertible if and only if
all $n$-minors of $A$ are invertible.
\item
An $LU$-factorization of a matrix
$A\in\mathcal{M}_\infty(\mathcal{A})$ exists if and only if it
exists for all $n$-minors of $A$. Moreover, in this case the
$LU$-factorization of $A$ is unique.
\end{enumerate}
\end{proposition}

\begin{definition}
Given a matrix $A\in\mathcal{M}_\infty(\mathcal{A})$ the
  \emph{shifted matrix} $s(A)$ of $A$ is
the diagonal blocked matrix formed by the identity
matrix of size $1\times 1$ and $A$, that is
\[
s(A)=
  \begin{pmatrix}
1     &    \\
          & A
\end{pmatrix}.
\]
\end{definition}

To apply an elementary row operation to a matrix $A$ is equivalent
to multiply $A$ on the left by a lower triangular matrix. For some
special matrices $A\in\mathcal{M}_\infty(\mathcal{A})$ it is
possible to perform the  elimination process of Gauss by using a
periodic sequence of elementary row operations, in the sense that
there exist a lower triangular matrix
$T_0\in\mathcal{M}_\infty(\mathcal{A})$, that would constitute the
fundamental part of the periodic sequence, such that the whole
periodic sequence of elementary row operations is of the form,
\[
T_0;\text{ next }s(T_0);\text{ next }s^2(T_0);\text{ next }s^3(T_0);\cdots.
\]
In other words, this means that the sequence of matrices
\[
A,\quad T_0A,\quad s(T_0)T_0A,\quad s^2(T_0)s(T_0)T_0A,\quad
s^3(T_0)s^2(T_0)s(T_0)T_0A, \quad\cdots \]
converges to an upper
triangular matrix.
This suggest the following definitions.

\begin{definition}
  Given  a lower triangular matrix $T_0\in\mathcal{M}_\infty(\mathcal{A})$
  the \emph{left iterated matrix}
  $T_0^{\text{L}}$ corresponding to $T_0$ is the infinite (from right to left) product
\[
T_0^{\text{L}}=\cdots  s^3(T_0)\,s^2(T_0)\,s(T_0)\,T_0.
\]
It is clear that this product converges to a lower triangular matrix.
Similarly, the \emph{right iterated matrix}
  $T_0^{\text{R}}$ corresponding to $T_0$ is the lower triangular matrix
given by the infinite (from left to right) product
\[
T_0^{\text{R}}=T_0\,s(T_0)\,s^2(T_0)\,s^3(T_0)\cdots.
\]
We say that a matrix $A\in\mathcal{M}_\infty(\mathcal{A})$
admits a \emph{periodic Gaussian elimination process} if there exist
a lower triangular matrix $T_0\in\mathcal{M}_\infty(\mathcal{A})$ such that
$T_0^{\text{L}}A$ is upper triangular.
In this case $T_0$ is called a \emph{fundamental sequence}
of the Gaussian elimination process of $A$.
\end{definition}
It is not difficult to prove that
 $T_0^{\text{L}}$ and $T_0^{\text{R}}$
are respectively characterized by the identities
\[
T_0^{\text{L}}=s(T_0^{\text{L}})\,T_0
\qquad\text{and}\qquad
T_0^{\text{R}}=T_0\,s(T_0^{\text{R}}),
\]
and it is clear that if $T_0$ is invertible then $T_0^{\text{L}}$ is
also invertible and
\[
\left(T_0^{\text{L}}\right)^{-1}=\left(T_0^{-1}\right)^{\text{R}}.
\]
Now the following proposition is clear.
\begin{proposition}
An infinite matrix $A$
admits a periodic gaussian elimination process if and only if
$A$ admits an LU-decomposition.
In fact, if $A=LU$ then the fundamental period is $T_0=s(L)L^{-1}$ and,
conversely, if $T_0$ is the fundamental period of $A$ then
\[
L^{-1}=\left(T_0\right)^{\text{L}}\text{ and }
L=\left(T_0^{-1}\right)^{\text{R}}.
\]
\end{proposition}

\subsection{Relevant infinite matrices}\label{subsec.ralevant inf. matrices}
In this subsection we introduce some infinite matrices that will be
used frequently in what follows and we also collect some nice
properties that these matrices enjoy.
\begin{enumerate}[.]
\item The Vandermonde matrix: \,$V_{ij}=i^{j-1}$.
\item The diagonal matrix formed by the powers of $q \in \mathcal{A}$:
 \,$(D_q)_{ij}=\delta_{ij}\;q^{i}$.
\item The diagonal matrix formed by the factorial numbers:
 \,$F_{ij}=\delta_{ij}\;(i-1)!$.
\item The lower triangular matrix formed by the Pascal numbers:
 \,$P_{ij}=\tbinom{i-1}{j-1}$.
\item The lower triangular matrix formed by the Stirling numbers of the second kind:
 \,$S_{ij}=S(i,j)$ for $i\ge j$, where
$S(i,j)=\textstyle\frac1{j!}\sum_{k=0}^j(-1)^k\tbinom{j}{k}(j-k)^i$.
\end{enumerate}

\

These matrices are,

\

\begin{tabular}{lll}
$V=
{
\left(
\begin{smallmatrix}
1\;     & 1\;    & 1\;    & 1\;    & \cdot\;\;      \\[1mm]
1\;     & 2\;    & 4\;    & 8\;    & \cdot\;\;      \\[1mm]
1\;     & 3\;    & 9\;    & 27\;   & \cdot\;\;      \\[1mm]
1\;     & 4\;    & 16\;   & 64\;   & \cdot\;\;      \\[1mm]
 .\;    & .    & .    & .  \\
\end{smallmatrix}
\right)
}$,
&
$D_q=
{
\left(
\begin{smallmatrix}
q     &         &         &         &         \\
      & q^2     &         &         &          \\
      &         & q^3     &         &           \\
      &         &         &q^4      &           \\
      &         &         &         & \cdot
\end{smallmatrix}
\right),
}$
&
$F=
{
\left(
\begin{smallmatrix}
1     &         &         &         &         \\
      & 1       &         &         &          \\
      &         & 2       &         &           \\
      &         &         &6        &           \\
      &         &         &         & \cdot
\end{smallmatrix}
\right),
}$
                                                    \\[1cm]
$P=
{
\left(
\begin{smallmatrix}
1\;     &        &        &        & \cdot\;\;      \\[1mm]
1\;     & 1\;    &        &        & \cdot\;\;      \\[1mm]
1\;     & 2\;    & 1\;    &        & \cdot\;\;      \\[1mm]
1\;     & 3\;    & 3\;    & 1\;    & \cdot\;\;      \\[1mm]
.\;     & .\;    & .\;    & .\;  \\
\end{smallmatrix}
\right)
}$,
&
$S=
{
\left(
\begin{smallmatrix}
1\;     &        &        &        &        & \cdot\;\;      \\[1mm]
1\;     & 1\;    &        &        &        & \cdot\;\;      \\[1mm]
1\;     & 3\;    & 1\;    &        &        & \cdot\;\;      \\[1mm]
1\;     & 7\;    & 6\;    & 1\;    &        & \cdot\;\;      \\[1mm]
1\;     & 15\;   & 25\;   & 10\;   &1\;     & \cdot\;\;      \\[1mm]
.\;     & .\;    & .\;    & .\;  \\
\end{smallmatrix}
\right). }$
                                                    \\[1cm]
\end{tabular}

\noindent Next we collect some of the properties of these matrices:

\begin{enumerate}[(1)]
\item\label{propert:Vandermonde} The Vandermonde matrix $V$ has an $LDU$-factorization.
It is given by $V=PFS^t$.

\item $P^{-1}=D_{-1}PD_{-1}$.

\item The classical recurrence relations
\[
\tbinom{i}{j}=\tbinom{i-1}{j-1}+\tbinom{i-1}{j}
\qquad\text{ and }\qquad
S(i,j)=S(i-1,j-1)+jS(i-1,j)
\]
that define, respectively, the Pascal and Stirling numbers
correspond to the matrix identities
$P=s(P)\,T_{0,P}$
and
$S=s(S)\,T_{0,S}$
where
\[T_{0,P}=\left(
\begin{smallmatrix}
1      &     &     &      & \cdot      \\
1      &1    &     &      & \cdot      \\
       &1    &1    &      & \cdot      \\
       &     &1    &1    & \cdot      \\
.      &.    &.    &.    &
\end{smallmatrix}
\right)
\qquad\text{ and }\qquad
T_{0,S}=\left(
\begin{smallmatrix}
1      &     &     &      & \cdot      \\
1      &1    &     &      & \cdot      \\
       &2    &1    &      & \cdot      \\
       &     &3    &1    & \cdot      \\
.      &.    &.    &.    &
\end{smallmatrix}
\right).
\]
Thus $P$ and $S$ are the left iterated matrices
$P=T_{0,P}^{\text{L}}$ and
$S=T_{0,S}^{\text{L}}$.
\item\label{propert:right}
$P$ and $S$ are also right iterated matrices.
Since
$P=
\left(
\begin{smallmatrix}
1      &     &     & \cdot      \\
1      &1    &     & \cdot      \\
1      &1    &1    & \cdot      \\
.      &.    &.    &
\end{smallmatrix}
\right)s(P)$
and $S=P\,s(S)$ we have
$P=\left(
\begin{smallmatrix}
1      &     &     & \cdot      \\
1      &1    &     & \cdot      \\
1      &1    &1    & \cdot      \\
.      &.    &.    &
\end{smallmatrix}
\right)^{\text{R}}$ and
$S=P^{\text{R}}$.
\item
Combining properties  (\ref{propert:Vandermonde}) and (\ref{propert:right})
we obtain that
the Vandermonde matrix $V=PFS^t$ admits a periodic Gaussian elimination process
with fundamental sequence
$\left(
\begin{smallmatrix}
1      &     &     & \cdot      \\
1      &1    &     & \cdot      \\
1      &1    &1    & \cdot      \\
.      &.    &.    &
\end{smallmatrix}
\right)^{-1}=\left(
\begin{smallmatrix}
\;\;\,1      &      &      &      & \cdot      \\
-1     &\;\;\,1    &      &      & \cdot      \\
       &-1    &\;\;\,1    &&       \cdot      \\
       &       &-1    &\;\;\,1    & \cdot      \\
\;\;\,.    & \;\;\,.      & \;\;\,.    &
\end{smallmatrix}
\right)$.
\item\label{id:change of basis}
Let  $p\in\mathcal{M}[t]$ be a polynomial and let $p_j$ and $a_j$
be, respectively, the coefficients of $p$ in terms of the bases
$\{t^j\}$ and $\{(t)_j\}$, that is
\begin{align*}
p(t)&=p_0+p_1t+p_2t^2+p_3t^3+\dots \\
    &=a_0(t)_0+a_1(t)_1+a_2(t)_2+a_3(t)_3+\dots.
\end{align*}
Then the coefficients $a_j$ and $p_j$ are related by
\[
\left(
\begin{smallmatrix}
a_0   \\
a_1   \\
a_2   \\
.
\end{smallmatrix}
\right)=
D_{-1}S^tD_{-1}P^t
\left(
\begin{smallmatrix}
p_0   \\
p_1   \\
p_2   \\
.
\end{smallmatrix}
\right)=
D_{-1}s(S)^tD_{-1}
\left(
\begin{smallmatrix}
p_0   \\
p_1   \\
p_2   \\
.
\end{smallmatrix}
\right).
\]

\end{enumerate}

\subsection{The matrix representation of the system $\eqref{eq:B}$}\label{subsec.MatrixRep}

Assume that $H$ and $E$ are
arbitrary elements of $\mathcal{A}$, y particular we do not assume that $[H,E]=cE$.

Recall that $p=p_0+p_1\,t+p_2\,t^2+p_3\,t^3+\dots$ satisfies the system
\begin{equation*}
E^np(H+n)=p(H-n)E^n \quad \text{for all }\, n\in\N,
\end{equation*}
if and only if the vector $(p_0, p_1, p_2, \dots)$ is a
solution of the following linear system,
\[
\begin{array}{*{6}{l@{}}@{=\hspace{-5pt}}c*{6}{l@{}}}
  E  p_0 &+& E  p_1(H\!+\!1) &+& E  p_2(H\!+\!1)^2 &\dots&
                    & p_{0}E   &+& p_{1}(H\!-\!1) E  &+& p_{2}(H\!-\!1)^2E &\dots  \\[2mm]
  E^2p_0 &+& E^2p_1(H\!+\!2) &+& E^2p_2(H\!+\!2)^2 &\dots&
                    & p_{0}E^2 &+& p_{1}(H\!-\!2) E^2&+& p_{2}(H\!-\!2)^2E^2&\dots \\[2mm]
  E^3p_0 &+& E^3p_1(H\!+\!3) &+& E^3p_2(H\!+\!3)^2 &\dots&
                    & p_{0}E^3 &+& p_{1}(H\!-\!3) E^3&+& p_{2}(H\!-\!3)^2E^3&\dots \\[2mm]
\vdots && \vdots && \vdots &&& \vdots&& \vdots&& \vdots
\end{array}
\]
or equivalently, that
\[
{
\left(
\begin{smallmatrix}
L_E   & L_ER_{H+1}   & L_ER_{H+1}^2     & L_ER_{H+1}^3    & \dots  \\[2mm]
L_E^2 & L_E^2R_{H+2} & L_E^2R_{H+2}^2   & L_E^2R_{H+2}^3 & \dots  \\[2mm]
L_E^3 & L_E^3R_{H+3} & L_E^3R_{H+3}^2   & L_E^3R_{H+3}^3 & \dots  \\[2mm]
\vdots    & \vdots   & \vdots    & \vdots
\end{smallmatrix}
\right)
}
{
\left(
\begin{smallmatrix}
p_0       \\[3mm]
p_1       \\[3mm]
p_2       \\[3mm]
 \vdots
\end{smallmatrix}
\right)
}=
{
\left(
\begin{smallmatrix}
R_E   & R_ER_{H-1}   & R_ER_{H-1}^2    & R_ER_{H-1}^3     & \dots  \\[2mm]
R_E^2 & R_E^2R_{H-2} & R_E^2R_{H-2}^2  & R_E^2R_{H-2}^3 & \dots  \\[2mm]
R_E^3 & R_E^3R_{H-3} & R_E^3R_{H-3}^2  & R_E^3R_{H-3}^3 & \dots  \\[2mm]
\vdots    & \vdots    & \vdots    & \vdots
\end{smallmatrix}
\right)
}
{
\left(
\begin{smallmatrix}
p_0       \\[3mm]
p_1       \\[3mm]
p_2       \\[3mm]
 \vdots
\end{smallmatrix}
\right), }
\]
where the matrices involved belong to
$\mathcal{M}_\infty(\End_{\k}(\mathcal{M}))$. Moreover, the above
identity holds if and only if $M{ \left(
\begin{smallmatrix}
p_0       \\
p_1       \\
p_2       \\
 .
\end{smallmatrix}
\right) }=0$, where

\[
M=
D_{L_E}
\left(
\begin{smallmatrix}
1\;   & R_{H+1} & R_{H+1}^2   & R_{H+1}^3 & \cdot  \\[1mm]
1\;   & R_{H+2} & R_{H+2}^2   & R_{H+2}^3 & \cdot  \\[1mm]
1\;   & R_{H+3} & R_{H+3}^2   & R_{H+3}^3 & \cdot  \\[1mm]
.     & .       & .           & .
\end{smallmatrix}
\right)
-
D_{R_E}
\left(
\begin{smallmatrix}
1\;   & R_{H-1} & R_{H-1}^2   & R_{H-1}^3 & \cdot  \\[1mm]
1\;   & R_{H-2} & R_{H-2}^2   & R_{H-2}^3 & \cdot  \\[1mm]
1\;   & R_{H-3} & R_{H-3}^2   & R_{H-3}^3 & \cdot  \\[1mm]
.     & .       & .           & .
\end{smallmatrix}
\right).
\]
Now, since
$R_{H\pm i}^{j-1}=\sum_{k\ge1}\binom{j-1}{k-1}(\pm i)^{k-1}R_H^{j-k}$ it follows that
\[
\left(
\begin{smallmatrix}
1\;   & R_{H+1} & R_{H+1}^2   & R_{H+1}^3 & \cdot  \\[1mm]
1\;   & R_{H+2} & R_{H+2}^2   & R_{H+2}^3 & \cdot  \\[1mm]
1\;   & R_{H+3} & R_{H+3}^2   & R_{H+3}^3 & \cdot  \\[1mm]
.     & .       & .           & .
\end{smallmatrix}
\right)=
VD_{1}P_{R_H}^t,\qquad
\left(
\begin{smallmatrix}
1\;   & R_{H-1} & R_{H-1}^2   & R_{H-1}^3 & \cdot  \\[1mm]
1\;   & R_{H-2} & R_{H-2}^2   & R_{H-2}^3 & \cdot  \\[1mm]
1\;   & R_{H-3} & R_{H-3}^2   & R_{H-3}^3 & \cdot  \\[1mm]
.     & .       & .           & .
\end{smallmatrix}
\right)=
- VD_{-1}P_{R_H}^t,
\]
where
\[
P_x
=\Big(\tbinom{i-1}{j-1}x^{i-j}\Big)=
{
\left(
\begin{smallmatrix}
1       &      &      &      & \cdot      \\
x      & 1    &      &      & \cdot      \\
x^2    & 2x   & 1    &      & \cdot      \\
x^3    & 3x^2 & 3x   & 1    & \cdot      \\
.       & .    & .    & .  \\
\end{smallmatrix}
\right).
}
\]
Hence $M=M_0P_{R_H}^t$ where
\[
M_0={
\left(
\begin{smallmatrix}
L_E  -R_E  \;\; &  L_E  + R_E  \;\; &  L_E  - R_E  \;\;   &   L_E  +  R_E  \;\; & \cdot\;\;  \\[3mm]
L_E^2-R_E^2\;\; & 2L_E^2+2R_E^2\;\; & 4L_E^2-4R_E^2\;\;   &  8L_E^2+ 8R_E^2\;\; & \cdot\;\;  \\[3mm]
L_E^3-R_E^3\;\; & 3L_E^3+3R_E^3\;\; & 9L_E^3-9R_E^3 \;\;  & 27L_E^3+27R_E^3\;\; & \cdot\;\;  \\[3mm]
      .         &       .           &       .             &       .
\end{smallmatrix}
\right).
}
\]

It is also convenient to consider the Pochhammer basis $\{(t)_j\}$
of $\k[t]$, instead of the basis $\{t^j\}$, to express the system
\eqref{eq:B}. If we do this, an analysis similar to
the previous one shows that
$p=a_0(t)_0+a_1(t)_1+a_2(t)_2+\dots\in{\mathcal{M}}[t]$ satisfies the system
\eqref{eq:B} if and only if
$M'{ \left(
\begin{smallmatrix}
a_0       \\
a_1       \\
a_2       \\
 .
\end{smallmatrix}
\right) }=0$, where
\[
M'=(M'_1-M'_2)F(P'_{R_H})^t,
\]
and the matrices $M'_1$, $M'_2$ and $P'_x$ are as follows,
$$P'_x
=\Big(\tbinom{i-1}{j-1}(x)_{i-j}\Big)=
{
\left(
\begin{smallmatrix}
1\;\;    &          &          &      & \cdot      \\[1.5mm]
(x)_1    & 1\;\;    &          &      & \cdot      \\[1.5mm]
(x)_2    & 2(x)_1   & 1\;\;    &      & \cdot      \\[1.5mm]
(x)_3    & 3(x)_2   & 3(x)_1   & 1\;\;& \cdot      \\[1.5mm]
.        & .    & .    & .  \\
\end{smallmatrix}
\right),
}
$$

\[
M'_1=D_{L_E}
\left(
\begin{smallmatrix}
1 \;\; & 1  \;\; & 1   \;\;  & 1        \;\; & \cdot\;\;  \\[1mm]
1 \;\; & 2  \;\; & 3   \;\;  &  4       \;\; & \cdot\;\;  \\[1mm]
1 \;\; & 3  \;\; & 6   \;\;  & 10       \;\; & \cdot\;\;  \\[1mm]
1 \;\; & 4  \;\; & 10  \;\;  & 20       \;\; & \cdot\;\;  \\[1mm]
      .         &       .           &       .             &       .
\end{smallmatrix}
\right) \quad \text{ and } \quad M'_2=D_{R_E} \left(
\begin{smallmatrix}
\;\;1    & -1   & \;\;     &       & \;\;     & \cdot      \\[1mm]
\;\;1    & -2   & \;\;1    &       & \;\;     & \cdot      \\[1mm]
\;\;1    & -3   & \;\;3    & -1    & \;\;     & \cdot      \\[1mm]
\;\;1    & -4   & \;\;6    & -4    & \;\;1    & \cdot      \\[1mm]
.       & .    & .    & .  \\
\end{smallmatrix}
\right).
\]
\begin{remark}
According to item \eqref{id:change of basis} of \S\ref{subsec.ralevant inf. matrices},
the matrices $M$ and $M'$ are related by
\[
M=M'D_{-1}s(S)^tD_{-1}.
\]
Thus, the matrix $M'$ could be seen as the result of factoring out from $M$
the Stirling numbers that the Vandermonde matrix (appearing in $M$) contains.
\end{remark}

\section{The LU-decomposition of the system $\eqref{eq:B}$}\label{subsec:LU}

\subsection{The L and U factors}\label{subsec:L and U factors}
We
have just seen that $M$ is the matrix corresponding to system \eqref{eq:B} in the canonical basis
and $M'$ is the  corresponding matrix in the Pochhammer basis.
We also have
\[
M=M_0P_{R_H}^t,\qquad
M'=(M'_1-M'_2)F(P'_{R_H})^t,\qquad
M=M'D_{-1}s(S)^tD_{-1}.
\]
Therefore, in order to find the
LU-decomposition of the system \eqref{eq:B} it is
sufficient to find the LU-decomposition of either $M_0$ or $M'_1-M'_2$.
Moreover, the entries of $M_0$ and $M'_1-M'_2$ are homogeneous polynomials
in the variables $L_{E}$ and $R_{E}$.
Since an homogeneous polynomial $q(x_1,x_2)$ in two variables
$x_1$ and $x_2$ is completely determined by the polynomial
$\tilde q(x)=q(x,1)$, we shall simplify the notation by substituting
the variables
\begin{equation}\label{eq:change}
L_{E}\text{ by }x\qquad\text{ and }\qquad R_{E}\text{ by }1.
\end{equation}
Under this transformations, the matrices $M_0$,  $M'_1$ and $M'_2$
become
\[
\tilde M_0=\left(
\begin{smallmatrix}
x   - 1   \;&\; x    + 1    \;&\; x    -  1   \;&\;  x   +   1   &\;\; \cdot  \\[3mm]
x^2 - 1   \;&\; 2x^2 + 2    \;&\; 4x^2 - 4    \;&\; 8x^2 +  8    &\;\; \cdot  \\[3mm]
x^3 - 1   \;&\; 3x^3 + 3    \;&\; 9x^3 - 9    \;&\;27x^3 + 27    &\;\; \cdot  \\[3mm]
.           & .         & .         & .
\end{smallmatrix}
\right),
\]
\[
\tilde M'_1=
D_x\left(
\begin{smallmatrix}
1 \;\; &  1    \;\; &  1     \;\;  &   1         \;\; & \cdot\;\;  \\[1mm]
1 \;\; & 2   \;\; & 3    \;\;  &  4        \;\; & \cdot\;\;  \\[1mm]
1 \;\; & 3   \;\; & 6    \;\;  & 10        \;\; & \cdot\;\;  \\[1mm]
1 \;\; & 4   \;\; & 10   \;\;  & 20        \;\; & \cdot\;\;  \\[1mm]
      .         &       .           &       .             &       .
\end{smallmatrix}
\right)
\quad \text{ and } \quad
\tilde M'_2=
\left(
\begin{smallmatrix}
\;\;1    & -1   & \;\;     &       & \;\;     & \cdot      \\[1mm]
\;\;1    & -2   & \;\;1    &       & \;\;     & \cdot      \\[1mm]
\;\;1    & -3   & \;\;3    & -1    & \;\;     & \cdot      \\[1mm]
\;\;1    & -4   & \;\;6    & -4    & \;\;1    & \cdot      \\[1mm]
.        & .    & .        & .     & .  \\
\end{smallmatrix}
\right).
\]
We point out that the LU-decompositions of the matrices
$\tilde M'_1$ and $\tilde M'_2$, separately, are
\[
\tilde M'_1=D_xPP^t \quad \text{ and } \quad
\tilde M'_2=P\left(
\begin{smallmatrix}
-1   & \;\;1 &      &      &       & \cdot      \\[1mm]
     & \;\;1 & -1   &      &       & \cdot      \\[1mm]
     & \;\;  & -1   & \;\;1&       & \cdot      \\[1mm]
     & \;\;  &      & \;\;1& -1    & \cdot      \\[1mm]
\;.  & \;\;. & \;.  & \;\;.& \;\;.  \\
\end{smallmatrix}
\right).
\]
The LU-decomposition of
$\tilde M'_1-\tilde M'_2$ is more subtle and it is obtained from Theorem \ref{thm:Eqinfinity}
as follows:
\begin{enumerate}[(1)]

\item $(\tilde M'_1-\tilde M'_2)F$ is the matrix corresponding to the system
\eqref{eq:B} for $H=0$, when it is expressed in terms of the basis $\{(t)_j\}$.

\item In other words, $(\tilde M'_1-\tilde M'_2)F$ is the matrix corresponding to the system
$\Eq^0$ for $H=0$, when it is expressed in terms of the basis $\{(t)_j\}$.

\item\label{item.T0}
The definition of the system $\Eq^k$ and the Recursion Formula
\ref{prop:recursion} imply that the system $\Eq^k$ is obtained by
keeping the first $k$ equations of the system $\Eq^{k-1}$ and by
incorporating the equations
\[
\eq^{k}_i=
\eq^{k-1}_i-2R_E(\eq^{k-1}_{i-1})+R_E^2(\eq^{k-1}_{i-2})
-\ad_E(\eq^{k-1}_{i-1})+R_E\ad_E(\eq^{k-1}_{i-2})
\]
for $i> k> 0$.
Applying the change of variables \eqref{eq:change} these equations are transformed
as follows,
\[
\eq^{k}_i=
\eq^{k-1}_i-\, (x+1)\eq^{k-1}_{i-1}\, +\, x\eq^{k-1}_{i-2}.
\]
This means that the lower triangular matrix that transforms
the system $\Eq^{k-1}$ into the system $\Eq^k$ is $s^{k-1}(T_0)$ with
\begin{equation}\label{eq:L0}
T_0=
\left(
\begin{smallmatrix}
1       &        &         &          & \cdot      \\
-x-1    &1       &         &          & \cdot      \\
x       &-x-1    &1        &          & \cdot      \\
        &x       &-x-1    &1          & \cdot      \\
.       & .      & .      & .
\end{smallmatrix}
\right).
\end{equation}

\item  Theorem \ref{thm:Eqinfinity} implies that the left iterated matrix $T_0^L$
associated to $T_0$ transforms the system $\Eq^{0}$ into the system $\Eq^\infty$.

\item\label{item:LU}
Let
\[
\tilde M'_1-\tilde M'_2=\tilde L'\tilde U'
\]
be the LU-decomposition of $\tilde M'_1-\tilde M'_2$.
According to the previous item, we have that
$\tilde L'=(T_0^L)^{-1}=(T_0^{-1})^R$ and
$\tilde U'$ is of the form $\tilde U'_1-\tilde U'_2$
where $\tilde U'_1F$ (respectively, $\tilde U'_2F$) is the matrix corresponding to the left
(respectively, right) hand side of the system $\Eq^\infty$ for $H=0$
(and thus $c=0$), when the system is expressed in terms of the basis
$\{(t)_j\}$.

\item According to Corollary \ref{coro:c=0,c=1}, when $H=0$, the $i$th equation of
$\Eq^\infty$ is
\begin{align*}
\eq_{i}^{i-1}:\sum_{r=0}^{i-1} \Big[\tbinom{i-1}{r}\,
\ad_E^{i-r}\partial_1^{i+r-1}\;p(i)&E^{r}
+\tbinom{i-1}{r}\ad_E^{i-r-1}\partial_1^{i+r-1}\;p(i)E^{r+1}\Big] \\
=&\sum_{r=0}^{i-1}\tbinom{i-1}{r}\,
\ad_E^{i-r-1}\partial_1^{i+r-1}\;p(i-2)E^{r\!+\!1}.
\end{align*}
Now, replacing $\ad_E$ by $(x-1)$ and $E$  by $1$ we obtain,
\begin{multline*}
\eq_{i}^{i-1}:\sum_{r=0}^{i-1}\tbinom{i-1}{r}\partial_1^{i+r-1}\;p(i)x(x-1)^{i-r-1} \\
=\sum_{r=0}^{i-1}\tbinom{i-1}{r}\,
\partial_1^{i+r-1}\;p(i-2)(x-1)^{i-r-1}.
\end{multline*}

\item
If $p=a_0(t)_0+a_1(t)_1+a_2(t)_2+\dots$ then, applying \eqref{eq:Pochhammer}, we obtain
\begin{multline*}
\eq_{i}^{i-1}:\sum_{j\ge 0}\sum_{r=0}^{j-i+1} \tbinom{i-1}{r}\,
(j-i-r+2)_{i+r-1}(i)_{j-i-r+1}x(x-1)^{i-r-1}a_j \\
=\sum_{j\ge 0}\sum_{r=0}^{j-i+1}\tbinom{i-1}{r}\,
(j-i-r+2)_{i+r-1}(i-2)_{j-i-r+1}(x-1)^{i-r-1}a_j.
\end{multline*}

\item\label{item:U} Now, the $(i,j)$ entry of the matrix $\tilde U'_1F$ (respectively, $\tilde U'_2F$) is
the coefficient multiplying $a_{j-1}$ in the left (respectively,
right) hand side of the $i$th equation the system  $\Eq^\infty$
given in (7), that is
\begin{align*}
(\tilde U'_1)_{ij}&=\frac{1}{(j-1)!}
 \sum_{r=0}^{j-i}\tbinom{i-1}{r}\,(j-i-r+1)_{i+r-1}(i)_{j-i-r}(x-1)^{i-r-1}x, \\
(\tilde U'_2)_{ij}&=\frac{1}{(j-1)!}
\sum_{r=0}^{j-i}\tbinom{i-1}{r}\,(j-i-r+1)_{i+r-1}(i-2)_{j-i-r}(x-1)^{i-r-1}. \notag
\end{align*}
\end{enumerate}

We shall now complete the description of the
LU-decomposition of the system \eqref{eq:B} by showing
that $\tilde L'$, $(\tilde L')^{-1}$, $\tilde U'_1$ and $\tilde U'_2$ are matrices whose
entries are given by some particular Jacobi polynomials.
The basic facts about Jacobi polynomials
that we shall need are collected in \S\ref{sec:Jacobi}.

\subsection{The entries of $L$ and $L^{-1}$ are ultraspherical Jacobi polynomials}
Recall that according to item\eqref{item:LU} of \S\ref{subsec:L and U factors}
we know that $\tilde L'=(T_0^{-1})^{\text{R}}$
and $(\tilde L')^{-1}=T_0^{\text{L}}$.
It is easy to see that the inverse of
\begin{equation*}
T_0=
\left(
\begin{smallmatrix}
1       &        &         &          & \cdot      \\
-x-1    &1       &         &          & \cdot      \\
x       &-x-1    &1        &          & \cdot      \\
        &x       &-x-1    &1          & \cdot      \\
.       & .      & .      & .
\end{smallmatrix}
\right)
\end{equation*}
is the lower triangular matrix
with entries $\left(T_0^{-1}\right)_{ij}=\frac{x^{i-j+1}-1}{x-1}$ if $i\ge j$, that is
\[
T_0^{-1}=
\left(
\begin{smallmatrix}
1          &        &         &          & \cdot      \\[1mm]
1+x        &1       &         &          & \cdot      \\[1mm]
1+x+x^2    &1+x     &1        &          & \cdot      \\[1mm]
1+x+x^2+x^3&1+x+x^2 &1+x      &1         & \cdot      \\[1mm]
.          & .      & .       & .
\end{smallmatrix}
\right).
\]

\begin{theorem}\label{thm:L0}
The entries of $\tilde L'$ and its inverse $(\tilde L')^{-1}$
are given by the following ultraspherical Jacobi polynomials.
\begin{align*}
(\tilde L')_{ij}
&=(-1)^{i-j}\;p_{i-j}^{-i,-i}(x), \\[2mm]
\left((\tilde L')^{-1}\right)_{ij}
&=(-1)^{i-j}\;\frac{j}{i}\;p_{i-j}^{j,j}(x).
\end{align*}
\end{theorem}

\begin{proof}
Let $K'$ and $K''$ be the matrices defined by the
right hand sides of the above equalities, that is
\[
K'_{ij}=(-1)^{i-j}\;p_{i-j}^{-i,-i}(x)\qquad\text{and}\qquad
K''_{ij}=(-1)^{i-j}\;\frac{j}{i}\;p_{i-j}^{j,j}(x).
\]
By the definition of left and right iterated matrices, it is sufficient to prove that
\[
K'=T_0^{-1}\,s(K')
\qquad\text{and}\qquad
K''=s(K'')\,T_0.
\]
The first identity holds if and only if
\[
(-1)^{i-j}\;p_{i-j}^{-i,-i}(x) = \sum_{k=j}^i
\dfrac{x^{i-k+1}-1}{x-1}\;(-1)^{k-j}\;p_{k-j}^{-k+1,-k+1}(x),
\]
or equivalently, if the following equality holds
\[
p_{i-j}^{-i,-i}(x)
=
\sum_{k=0}^{i-j} \dfrac{x^{i-j-k+1}-1}{x-1}\;(-1)^{k+j-i}\;p_{k}^{-k-j+1,-k-j+1}(x),
\]
which in turn holds if and only if
\[
P_{i-j}^{-i,-i}\left(\tfrac{x+1}{x-1}\right) = \sum_{k=0}^{i-j}
\left(1-x^{i-j-k+1}\right)\;
(1-x)^{k+j-i-1}P_{k}^{-k-j+1,-k-j+1}\left(\tfrac{x+1}{x-1}\right).
\]
Now, this last identity is equivalent to
\[
P_{i-j}^{-i,-i}(x)
=
\sum_{k=0}^{i-j+1}
\left(\left(\tfrac{1-x}2\right)^{i-j-k+1}-\left(\tfrac{-1-x}2\right)^{i-j-k+1}\right)\;
P_{k}^{-k-j+1,-k-j+1}(x),
\]
and this last equality is true, since it is the difference of the
two identities stated in Lemma \ref{lemma:ALInv} for $n=i-j+1$ and
$\alpha=i$. This completes the proof of the identity
$K'=T_0^{-1}\,s(K').$

The second identity holds if and only if
\[
(-1)^{i-j}\;\frac{j}{i}\;p_{i-j}^{j,j}(x)
=
\sum_{k=j}^i (-1)^{i-k}\;\frac{k-1}{i-1}\;p_{i-k}^{k-1,k-1}(x)\;(T_0)_{kj}
\]
for $i\ge2$, or equivalently if
\[
\frac{j}{i}\;p_{i-j}^{j,j}(x)
=
\frac{j-1}{i-1}\;p_{i-j}^{j-1,j-1}(x)+
\frac{j(x+1)}{i-1}\;p_{i-j-1}^{j,j}(x)+
\frac{(j+1)x}{i-1}\;p_{i-j-2}^{j+1,j+1}(x),
\]
which in turn holds if and only if
\begin{multline*}
\frac{j}{i}\;P_{i-j}^{j,j}\left(\tfrac{x+1}{x-1}\right) \\
=
\tfrac{j-1}{i-1}\;P_{i-j}^{j-1,j-1}\left(\tfrac{x+1}{x-1}\right)+
\tfrac{j}{i-1}\tfrac{x+1}{x-1}\;P_{i-j-1}^{j,j}\left(\tfrac{x+1}{x-1}\right)+
\tfrac{j+1}{i-1}\tfrac{x}{(x-1)^2}\;P_{i-j-2}^{j+1,j+1}\left(\tfrac{x+1}{x-1}\right),
\end{multline*}
or equivalently, if the following equality holds
\[ 4(i-1)j\;P_{i-j}^{j,j}(x) =
4i(j-1)\;P_{i-j}^{j-1,j-1}(x)+ 4ijx\;P_{i-j-1}^{j,j}(x)+
i(j+1)(x^2-1)\;P_{i-j-2}^{j+1,j+1}(x).
\]
Now, this last identity is true since it is $(1+j)$ times equation
\eqref{id:4} with $n=i-j-1$ and $\alpha=j+1$, plus $2(j-1)$ times
equation \eqref{id:3.1} with $n=i-j$ and $\alpha=j$. This completes
the proof of the identity $K''=s(K'')\,T_0.$
\end{proof}

\begin{remark}
It is known that the ultraspherical Jacobi polynomials constitute a
2-parameter family of polynomials that are orthogonal with respect
to a continuous measure. The above theorem provides the following
``discrete orthogonality'' relationship that involves once many of them,
\[
\left(
\begin{smallmatrix}
\frac11 p_0^{1,1} &                  &                 &                  & \cdot   \\[1mm]
\frac12 p_1^{1,1} &\frac22 p_0^{2,2} &                 &                  & \cdot   \\[1mm]
\frac13 p_2^{1,1} &\frac23 p_1^{2,2} &\frac33 p_0^{3,3}&                  & \cdot   \\[1mm]
\frac14 p_3^{1,1} &\frac24 p_2^{2,2} &\frac34 p_1^{3,3}&\frac11 p_0^{4,4} & \cdot   \\[1mm]
.                 & .                & .               & .
\end{smallmatrix}
\right)
\left(
\begin{smallmatrix}
p_0^{-1,-1} &             &            &             & \cdot   \\[1.23mm]
p_1^{-2,-2} & p_0^{-2,-2} &            &             & \cdot   \\[1.23mm]
p_2^{-3,-3} & p_1^{-3,-3} & p_0^{-3,-3}&             & \cdot   \\[1.23mm]
p_3^{-4,-4} & p_2^{-4,-4} & p_1^{-4,-4}& p_0^{-4,-4} & \cdot   \\[1.23mm]
.           & .           & .          & .
\end{smallmatrix}
\right)=
\left(
\begin{smallmatrix}
1\:\: &    \:\:&  \:\:   &  \:\: & \cdot    \\[2.20mm]
 \:\: & 1  \:\:&  \:\:   &  \:\: &  \cdot   \\[2.20mm]
 \:\: &    \:\:& 1\:\:   &  \:\: &  \cdot   \\[2.20mm]
 \:\: &    \:\:&  \:\:   & 1\:\: &  \cdot   \\[2.20mm]
.\:\: & .  \:\:& .\:\:   & .
\end{smallmatrix}
\right).
\]
\end{remark}

\subsection{The entries of $\tilde U'_1$ and $\tilde U'_2$ are Jacobi polynomials}

\begin{theorem}
The entries of $\tilde U'_1$ and $\tilde U'_2$
are given by the following Jacobi polynomials.
\begin{align*}
\left(\tilde U'_1\right)_{ij}
&=(-1)^{j-i} \,(x-1)^{2i-j-1}\,x\,p_{j-i}^{-j,0}(x), \\[2mm]
\left(\tilde U'_2\right)_{ij}
&=(-1)^{j-i} \,(x-1)^{2i-j-1}\,p_{j-i}^{-j+2,-2}(x).
\end{align*}
\end{theorem}
\begin{proof}
The theorem will be proved if we show that
\begin{align*}
\tilde U'_1\left(\tfrac{x+1}{x-1}\right)_{ij}&=(-1)^{j-i}\,
\left(\tfrac{x-1}{2}\right)^{1-i}\,\tfrac{x+1}{x-1}\, P_{j-i}^{-j,0}(x), \\
\intertext{and}
\tilde U'_2\left(\tfrac{x+1}{x-1}\right)_{ij}&=(-1)^{j-i}\,
\left(\tfrac{x-1}{2}\right)^{1-i}\,P_{j-i}^{-j+2,-2}(x).
\end{align*}
In fact, according to item \eqref{item:U}
\begin{align*}
\tilde U'_1&\left(\tfrac{x+1}{x-1}\right)_{ij}=
 \tfrac{x+1}{2(j-1)!}\sum_{r=0}^{j-i}\tbinom{i-1}{r}\,(j-i-r+1)_{i+r-1}(i)_{j-i-r}
                            \left(\tfrac{x-1}2\right)^{r-i} \\
&=
\tfrac{x+1}2\left(\tfrac{x-1}2\right)^{-i}\tfrac{1}{(j-i)!}
\sum_{r=0}^{j-i}\tbinom{j-i}{r}\,(i-r)_{r}(i)_{j-i-r}\left(\tfrac{x-1}2\right)^{r}  \\
&=
\tfrac{x+1}2\left(\tfrac{x-1}2\right)^{-i}\tfrac{(-1)^{j-i}}{(j-i)!}
\sum_{r=0}^{j-i}\tbinom{j-i}{r}\,(1-i)_{r}(1-j+r)_{j-i-r}\left(\tfrac{x-1}2\right)^{r}  \\
&=
(-1)^{j-i}\left(\tfrac{x-1}2\right)^{1-i}
\tfrac{x+1}{x-1}P_{j-i}^{-j,0}(x).
\end{align*}
Analogously,
\begin{align*}
\tilde U'_2&\left(\tfrac{x+1}{x-1}\right)_{ij}=\tfrac1{(j-1)!}
\sum_{r=0}^{j-i}\tbinom{i-1}{r}\,(j-i-r+1)_{i+r-1}(i-2)_{j-i-r}
                    \left(\tfrac{x-1}2\right)^{r-i+1} \\
&=\left(\tfrac{x-1}2\right)^{1-i}\tfrac{1}{(j-i)!}
\sum_{r=0}^{j-i}\tbinom{j-i}{r}\,(i-r)_{r}(i-2)_{j-i-r}\left(\tfrac{x-1}2\right)^{r}  \\
&=\left(\tfrac{x-1}2\right)^{1-i}\tfrac{(-1)^{j-i}}{(j-i)!}
\sum_{r=0}^{j-i}\tbinom{j-i}{r}\,(1-i)_{r}(3-j+r)_{j-i-r}\left(\tfrac{x-1}2\right)^{r}  \\
&=
(-1)^{j-i}\left(\tfrac{x-1}2\right)^{1-i}P_{j-i}^{-j+2,-2}(x).
\end{align*}
This completes the proof of the theorem.
\end{proof}
We have therefore proved the following theorem
\begin{theorem}
The LU-decomposition of
\[
\tilde M'_1-\tilde M'_2=\left(
\begin{smallmatrix}
x   - 1   \;&\; x    + 1    \;&\; x       \;&\;  x      &\;\; \cdot  \\[3mm]
x^2 - 1   \;&\; 2x^2 + 2    \;&\; 3x^2 - 1    \;&\; 4x^2     &\;\; \cdot  \\[3mm]
x^3 - 1   \;&\; 3x^3 + 3    \;&\; 6x^3 - 3    \;&\;10x^3 + 1    &\;\; \cdot  \\[3mm]
.           & .         & .         & .
\end{smallmatrix}
\right)
\]
is
$
\tilde M'_1-\tilde M'_2=\tilde L'\tilde U'
$
where
\[
\begin{array}{rclr}
\displaystyle (\tilde L')_{ij}&=&\displaystyle (-1)^{i-j}\;p_{i-j}^{-i,-i}(x), &i\ge j; \\[3mm]
\displaystyle (\tilde U')_{ij}&=&
\displaystyle (-1)^{j-i} \,(x-1)^{2i-j-1}\,\left( x\,p_{j-i}^{-j,0}(x) - p_{j-i}^{-j+2,-2}(x)\right), &i\le j.
\end{array}
\]
\end{theorem}

\begin{corollary}
The LU-decomposition of
\[
\tilde M_0=\left(
\begin{smallmatrix}
x   - 1   \;&\; x    + 1    \;&\; x    -  1   \;&\;  x   +   1   &\;\; \cdot  \\[3mm]
x^2 - 1   \;&\; 2x^2 + 2    \;&\; 4x^2 - 4    \;&\; 8x^2 +  8    &\;\; \cdot  \\[3mm]
x^3 - 1   \;&\; 3x^3 + 3    \;&\; 9x^3 - 9    \;&\;27x^3 + 27    &\;\; \cdot  \\[3mm]
.           & .         & .         & .
\end{smallmatrix}
\right)
\]
is
\[
\tilde M_0=\tilde L_0 \tilde U_0
\left(
\begin{smallmatrix}
1     &         &         &         &         &         \\
      & 1       &         &         &         &          \\
      &         & 2       &         &         &           \\
      &         &         &6        &         &           \\
      &         &         &         & 24      &             \\
      &         &         &         &         & \cdot
\end{smallmatrix}
\right)
{\tiny
\left(
\begin{array}{rrrrrr}
1       &-1      & 1    &-1    & 1      & \cdot      \\[1mm]
        & 1      &-3    & 7    &-15     & \cdot      \\[1mm]
        &        & 1    &-6    & 25     & \cdot      \\[1mm]
        &        &      & 1    &-10     & \cdot      \\[1mm]
        &        &      &      & 1      & \cdot      \\[1mm]
.       & .      & .    & .    & .      \\
\end{array}
\right)
}
\]
where
\[
\begin{array}{rclr}
\displaystyle (\tilde L_0)_{ij}&=&\displaystyle (-1)^{i-j}\;p_{i-j}^{-i,-i}(x), &i\ge j; \\[3mm]
\displaystyle (\tilde U_0)_{ij}&=&\displaystyle (-1)^{i-j}\frac{j}{i}\;(x-1)^{2i-j}\;p_{j-i}^{-j,-1}(x), &i\le j;\\[2mm]
\end{array}
\]
and the explicit numerical
matrices are respectively formed by the factorial and
(up to the minus signs) Stirling numbers of the second kind.
\end{corollary}

\begin{proof}
We know that
\begin{align*}
\tilde M_0  &=(\tilde M'_1-\tilde M'_2)FD_{-1}s(S)^tD_{-1} \\
            &=\tilde L'\tilde U'FD_{-1}s(S)^tD_{-1}.
\end{align*}
This implies that $\tilde L'$ is the L-part of $\tilde M_0$ and thus we obtain
the expression for  $(\tilde L_0)_{ij}$.
Now we must prove that
\[
\tilde U_0 FD_{-1}S^tD_{-1}=\tilde U'FD_{-1}s(S)^tD_{-1}
\]
which is easily seen to be equivalent to prove that
\[
\tilde U_0
{\tiny
\left(
\begin{array}{rrrrrr}
1       &-1      &      &  \cdot      \\[1mm]
        & 1      &-1    &  \cdot      \\[1mm]
        &        & 1    &  \cdot      \\[1mm]
.       & .      & .    &        \\
\end{array}
\right)
}
=\tilde U'
\]
that is
\[
(x-1)\left(
\frac{j}{i}\;\;p_{j-i}^{-j,-1}(x)+
\frac{j-1}{i}\;(x-1)\;p_{j-i-1}^{-j+1,-1}(x)
\right)
=x\,p_{j-i}^{-j,0}(x) - p_{j-i}^{-j+2,-2}(x)
\]
or
\[
2\left(
\frac{j}{i}\;\;P_{j-i}^{-j,-1}(x)+
\frac{j-1}{i}\;\;P_{j-i-1}^{-j+1,-1}(x)
\right)
=(x+1)\,P_{j-i}^{-j,0}(x) - (x-1)P_{j-i}^{-j+2,-2}(x)
\]
It is now straightforward to derive this identity by using
the formulas (4.5.1) and (4.5.4) from \cite{Sz}.
\end{proof}

\section{Appendix: Jacobi and Gegenbauer polynomials}\label{sec:Jacobi}
The \emph{Jacobi} polynomials $P^{\alpha,\beta}_n$ are defined for
non negative integers $n$ and arbitrary (rational) numbers $\alpha$ and
$\beta$ as follows (see Chapter IV in \cite{Sz}),
\begin{multline*}
P^{\alpha,\beta}_n(x)=
\frac1{n!}\sum_{k=0}^n
\binom{n}{k}(n+\alpha+\beta+1)\dots(n+\alpha+\beta+k) \\[-2mm]
\times(\alpha+k+1)\dots(\alpha+n)\left(\frac{x-1}{2}\right)^k,
\end{multline*}
with the understanding that the general coefficient
\[\binom{n}{k}(n+\alpha+\beta+1)\dots(n+\alpha+\beta+k)
(\alpha+k+1)\dots(\alpha+n)\] is equal to
$(\alpha+1)\dots(\alpha+n)$ if $k=0$, and equal to
$(n+\alpha+\beta+1)\dots(2n+\alpha+\beta)$ if $k=n$. The Jacobi
polynomials can also be represented as
\[
P^{\alpha,\beta}_n(x)= \frac{(\alpha+1)\dots(n+\alpha)}{n!}\;
{}_2F_1\left(-n,n+\alpha+\beta+1;\alpha+1;\frac{1-x}{2}\right),
\]
where ${}_2F_1$ is the hypergeometric function of Gauss. If
$\alpha=\beta$ the normalized Jacobi polynomials
\[
\frac{\Gamma(\alpha+1)}{\Gamma(2\alpha+1)}
\frac{\Gamma(n+2\alpha+1)}{\Gamma(n+\alpha+1)}\;P^{\alpha,\alpha}_n(x)
\]
are called \emph{ultraspherical} polynomials or \emph{Gegenbauer}'s
polynomials. In what follows we shall call \emph{ultraspherical
Jacobi polynomials} to the ``unnormalized'' Gegenbauer's polynomials
$P^{\alpha,\alpha}_n(x).$

Let us consider
\[
p_n^{\alpha,\beta}(x)=(x-1)^nP^{\alpha,\beta}_n\left(\tfrac{x+1}{x-1}\right),
\]
it is clear that $p_n^{\alpha,\beta}(x)$ is again a polynomial.
These polynomials can be expressed in terms of the hypergeometric
function of Gauss as follows (see (4.22.1) in \cite{Sz}),
\[
p_n^{\alpha,\beta}(x)=
\frac{(n+\alpha+\beta+1)\dots(2n+\alpha+\beta)}{n!}\;
{}_2F_1\left(-n,-n-\alpha;-2n-\alpha-\beta;x-1\right).
\]

\subsection{Polynomial identities}
It is well known that
\begin{equation}\label{id:2}
P_n^{\alpha,\beta}(x)=(-1)^nP_n^{\beta,\alpha}(-x).
\end{equation}
In particular, the ultraspherical Jacobi polynomials are even or odd
according as $n$ is even or odd.
The ultraspherical Jacobi polynomials also satisfy
the following identities (see (4.7.14) and (4.7.28) in \cite{Sz}),
\begin{align}\label{id:3.1}
(n+2\alpha)P_n^{\alpha,\alpha}(x)
&=2(n+\alpha)P_n^{\alpha-1,\alpha-1}(x)+x(n+\alpha)P_{n-1}^{\alpha,\alpha}(x) \\
\label{id:3.2}
x(n+2\alpha)P_{n}^{\alpha,\alpha}(x)
&=2(n+1)P_{n+1}^{\alpha-1,\alpha-1}(x)+(n+\alpha)P_{n-1}^{\alpha,\alpha}(x).
\end{align}
The difference between these two identities gives,
\begin{multline}\label{id:1}
(1-x)(n+2\alpha)P_{n}^{\alpha,\alpha}(x) \\
=-2(n+1)P_{n+1}^{\alpha-1,\alpha-1}(x)+2(n+\alpha)P_n^{\alpha-1,\alpha-1}(x)
-(1-x)(n+\alpha)P_{n-1}^{\alpha,\alpha}(x).
\end{multline}
On the other hand, multiplying by $x$ equation \eqref{id:3.1} and
subtracting from it equation \eqref{id:3.2} yields the following
identity,
\begin{equation}\label{id:4}
2(n+1)P_{n+1}^{\alpha-1,\alpha-1}(x)
=2x(n+\alpha)P_{n}^{\alpha-1,\alpha-1}(x)+(x^2-1)(n+\alpha)P_{n-1}^{\alpha,\alpha}(x).
\end{equation}

\begin{lemma}\label{lemma:ALInv}
The ultraspherical Jacobi polynomials satisfy,
\begin{align*}
2\sum_{k=0}^n\left(\tfrac{1-x}{2}\right)^{n-k}P_{k}^{\alpha+n-k,\alpha+n-k}(x)
&=\frac{n+2\alpha}{n+\alpha}\;P_{n}^{\alpha,\alpha}(x)+P_{n-1}^{\alpha,\alpha}(x) \\
2\sum_{k=0}^n\left(\tfrac{-1-x}{2}\right)^{n-k}P_{k}^{\alpha+n-k,\alpha+n-k}(x)
&=\frac{n+2\alpha}{n+\alpha}\;P_{n}^{\alpha,\alpha}(x)-P_{n-1}^{\alpha,\alpha}(x)
\end{align*}
\end{lemma}
\begin{proof}
The second identity follows immediately from the first one by using
\eqref{id:2}, then we shall prove the first identity by induction on
$n$. It is clear that this identity holds for $n=0$ and all
$\alpha$, hence if we assume that it holds for $n$ and all $\alpha$
we have,
\begin{align*}
2\sum_{k=0}^{n+1}&\left(\tfrac{1-x}{2}\right)^{n+1-k}P_{k}^{\alpha+n+1-k,\alpha+n+1-k}(x) \\[1mm]
&=
\frac{1-x}{2}\,2
\sum_{k=0}^{n}\left(\tfrac{1-x}{2}\right)^{n-k}P_{k}^{\alpha+n+1-k,\alpha+n+1-k}(x)
+2P_{n+1}^{\alpha,\alpha}(x)  \\[2mm]
&=
\frac{1-x}{2}\;
\frac{n+2\alpha+2}{n+\alpha+1}\;P_{n}^{\alpha+1,\alpha+1}(x)
+\frac{1-x}{2}\;P_{n-1}^{\alpha+1,\alpha+1}(x)
+2P_{n+1}^{\alpha,\alpha}(x) \\[3mm]
&=
-\frac{n+1}{n+\alpha+1}\;P_{n+1}^{\alpha,\alpha}(x)+P_n^{\alpha,\alpha}(x)
+2P_{n+1}^{\alpha,\alpha}(x)  \\[3mm]
&=
\frac{n+1+2\alpha}{n+1+\alpha}\;P_{n+1}^{\alpha,\alpha}(x)+P_n^{\alpha,\alpha}(x),
\end{align*}
which completes the proof of the first identity. We point out that
in the last but one equality we used the identity \eqref{id:1} with
$\alpha+1$ instead of $\alpha$.
\end{proof}


\end{document}